\documentclass{article}

\usepackage[utf8]{inputenc}
\usepackage[italian,english]{babel}

\usepackage{geometry}	
\usepackage{fancyhdr}	
\usepackage{titlesec}		
\usepackage{indentfirst}  	

\usepackage[linktoc=all]{hyperref}	
\usepackage{tikz}				

\usepackage{amsmath}
\usepackage{amssymb}	
\usepackage{mathrsfs}	
\usepackage{amsthm}	

\theoremstyle{plain}
\newtheorem{mythm}{Theorem}[section]
\newtheorem{myprop}[mythm]{Proposition}
\newtheorem{mylemma}[mythm]{Lemma}
\newtheorem{mycor}[mythm]{Corollary}
\newtheorem{mydef}[mythm]{Definition}

\newtheorem{introthm}{Theorem}

\theoremstyle{remark}
\newtheorem*{myex}{Example}
\newtheorem*{myrmk}{Remark}

\newcommand{\sgr}{\le}

\newcommand{\gen}[1]{\langle#1\rangle}

\newcommand{\ol}[1]{\overline{#1}}

\newcommand{\rar}{\rightarrow}

\newcommand{\restr}[2]{#1\big|_{#2}}

\newcommand{\core}[1]{\text{core}(#1)}
\newcommand{\bcore}[1]{\text{core}_*(#1)}
\newcommand{\cov}[1]{\text{cov}(#1)}
\newcommand{\grafo}{F_n\text{-labeled graph}}
\newcommand{\grafos}{F_n\text{-labeled graphs}}
\newcommand{\rank}[1]{\text{rank}(#1)}
\newcommand{\im}[1]{\text{im}(#1)}
\newcommand{\fold}[1]{\text{fold}(#1)}
\newcommand{\subd}[2]{\text{subd}_{#1}(#2)}

\newcommand{\FF}[1]{\textit{FF}_{#1}}
\newcommand{\fine}{\text{fine}}

\title{A fine property of Whitehead's algorithm}
\author{
Dario Ascari \thanks{Ascari was funded by the Engineering and Physical Sciences Research Council.}\\
{\small \textit{Mathematical Institute, Andrew Wiles Building,}}\\
{\small \textit{University of Oxford, Oxford OX2 6GG, UK}}\\
{\small e-mail: \texttt{ascari@maths.ox.ac.uk}}\\
\\
}

\begin{document}

\maketitle

\begin{abstract}
We develop a refinement of Whitehead's algorithm for primitive words in a free group. We generalize to subgroups, establishing a strengthened version of Whitehead's algorithm for free factors. We make use of these refinements in proving new results about primitive elements and free factors in a free group. These include a relative version of Whitehead's algorithm, and a criterion that tests whether a subgroup is a free factor just by looking at its primitive elements. We develop an algorithm to determine whether or not two vertices in the free factor complex have distance $d$ for $d=1,2,3$, as well as $d=4$ in a special case.
\end{abstract}

\begin{center}
\end{center}


\section{Introduction}

An algorithm to determine whether an element of a free group is primitive or not was first found by Whitehead in 1936; it is based on the following theorem:

\begin{introthm}[Whitehead]\label{introWhitehead}
Let $w$ be a cyclically reduced word, which is primitive but not a single letter. Then there is a Whitehead automorphism $\varphi$ such that the cyclic length  of $\varphi(w)$ is strictly smaller than the cyclic length of $w$.
\end{introthm}

Theorem \ref{introWhitehead} has been proved and studied in several different ways over the years. Whitehead's original proof involves working with three-manifolds, performing surgery on embedded paths and surfaces, see \cite{Whitehead}. Of particular importance are the peak-reduction techniques introduced by Rapaport in \cite{Rapaport}; see also \cite{Lyndon} for a simplified version of the argument. Another surprisingly short proof, based on Stallings' folding operations, appeared recently in \cite{Heusener}.

\

Let $x_1,...,x_n$ be a fixed basis for $F_n$. We recall that a Whitehead automorphism $\varphi$ is an automorphism such that, for some $a\in\{x_1,...,x_n,\ol{x_1},...,\ol{x_n}\}$, we have $\varphi(a)=a$ and $\varphi(x_j)\in\{x_j,ax_j,x_j\ol{a},ax_j\ol{a}\}$ for each other generator $x_j\not=a,\ol{a}$. A generic element $w\in F_n$ consists of a (reduced) sequence of symbols in $\{x_1,...,x_n,\ol{x_1},...,\ol{x_n}\}$, and in order to obtain the image $\varphi(w)$, we can just apply $\varphi$ letter by letter to the sequence of symbols (and then reduce the resulting word). In the present paper, we shall build on the following refinement of Whitehead's theorem \ref{introWhitehead}.

\begin{introthm}\label{introfinesse}
The automorphism in theorem \ref{introWhitehead} can be chosen in such a way that every letter $a$ or $\ol a$, that is added when we apply $\varphi$ to $w$ letter by letter, immediately cancels (in the cyclic reduction process).
\end{introthm}

Theorem \ref{introfinesse} can be deduced fairly directly from Whitehead's original argument \cite{Whitehead}; however, we were not able to find this statement in literature. It can also be derived with the techniques of \cite{Heusener}, as will be shown in the body of the present paper. It is difficult to imagine how theorem \ref{introfinesse} might be proved with the peak-reduction techniques of \cite{Rapaport}.

\

We make use of theorems \ref{introWhitehead} and \ref{introfinesse} to prove the following theorem about free factors in a free group:

\begin{introthm}\label{introprimitiveinfreefactors}
Let $H\sgr F_n$ be a finitely generated subgroup. Suppose that every element of $H$ which is primitive in $H$ is also primitive in $F_n$. Then $H$ is a free factor.
\end{introthm}

We point out that the additional property of Whitehead's algorithm really plays a key role in the proof of theorem \ref{introprimitiveinfreefactors}.

\

We also generalize the fine property of theorem \ref{introfinesse} to subgroups. Generalizations of Whitehead's algorithm for free factors already exist, as shown in \cite{Gersten}, but the known proofs are based on the peak-reduction techniques, which, as we have noted, do not seem appropriate to our refinement. We here introduce the concept of Whitehead graph for a subgroup, and we generalize the ideas of \cite{Heusener} in order to get a refined statement of Whitehead's algorithm for subgroups. The standard statement is the following theorem \ref{introWhitehead2} and we add the fine property in theorem \ref{introfine}.

\begin{introthm}\label{introWhitehead2}
Let $H\sgr F_n$ be a free factor and suppose $\core{H}$ has more than one vertex. Then there is a Whitehead automorphism $\varphi$ such that $\core{\varphi(H)}$ has strictly fewer vertices and strictly fewer edges than $\core{H}$.
\end{introthm}

\begin{introthm}\label{introfine}
The automorphism in theorem \ref{introWhitehead2} can be chosen in such a way that $\core{\varphi(H)}$ can be obtained from $\core{H}$ by means of a quotient that collapses some of the edges to points whilst preserving the labels and orientations on the other edges.
\end{introthm}

This additional property turns out to have several interesting features, and in particular it behaves well with respect to subgroups (see lemmas \ref{finesubgroups} and \ref{fineimage} in the body of the paper). It also allows us to deduce a relative version of Whitehead's algorithm. Let $x_1,...,x_n$ be a fixed basis for $F_n$.

\begin{introthm}\label{introWhitehead4}
Let $w\in F_n$ be a primitive element that is not a single letter. Suppose there is an automorphism $\theta:F_n\rar F_n$ such that $\theta(\gen{x_1,...,x_k})=\gen{x_1,...,x_k}$ and $\theta(w)=x_{k+1}$. Then there is a Whitehead automorphism $\varphi$ such that:

(i) $\varphi(x_i)=x_i$ for $i=1,...,k$.

(ii) The length of $\varphi(w)$ is strictly smaller than the length of $w$.

(iii) Every letter, which is added to $w$ when applying $\varphi$ to $w$ letter by letter, immediately cancels (in the free reduction process).
\end{introthm}

\

We then use the techniques developed in order to investigate the structure of the free factor complex $\FF{n}$. The free factor complex of a free group is the analogous to the curve complex of a surface. A famous rigidity theorem of Ivanov states that the isometries of the curve complex are essentially the mapping class group of the surface; in the same way, there is a rigidity theorem due to Bestvina and Bridson stating that the isometries of $\FF{n}$ are essentially the outer automorphisms of $F_n$, see \cite{Bridson}. Just as the curve complex turned out to be hyperbolic, $\FF{n}$ is hyperbolic too: in \cite{Bestvina}, Bestvina and Feighn study the large-scale geometry of $\FF{n}$ and give a detailed description of lines which are geodesics up to a reparametrization and up to distance $C$. Unfortunately, the constant $C$ is quite large, so those techniques do not give much information about the local geometry of $\FF{n}$. In the present paper, we investigate the existence of an algorithm for the computation of the exact distance in $\FF{n}$. There are algorithms to compute the distance between any two points in the curve complex of a surface, but the same question for $\FF{n}$ remains open. We are able to determine whether two vertices are at distance $d$ for $d=1,2,3$; we are also able to do that for $d=4$ in the particular case when one of the two vertices represents a conjugacy class of free factors of rank $n-1$. In order to do this, we make use of the tools developed earlier in this paper, together with an idea which appeared in \cite{Goldstein}.

\section*{Acknowledgements}

I would like to thank my supervisor Prof. Martin R. Bridson for having given me the interesting theorem \ref{introprimitiveinfreefactors} to think about, and for all the suggestions he gave me when I was stuck while trying to solve it. I also thank him for all the help he gave me while I was thinking about the free factor complex.

\section{Preliminaries and notations}\label{Preliminaries}

We work inside a finitely generated free group $F_n$ of rank $n$, generated by $x_1,...,x_n$. We write $\ol{x_i}=x_i^{-1}$. We denote with $R_n$ the standard $n$-rose, i.e. the graph with one vertex $*$ and $n$ oriented edges labeled $x_1,...,x_n$. The fundamental group $\pi_1(R_n,*)$ will be identified with $F_n$: the path going along the edge labeled $x_i$ (with the right orientation) corresponds to the element $x_i\in F_n$.

\begin{mydef}\label{grafo}
An \textbf{$\grafo$} is a graph $G$ together with a map $f:G\rar R_n$ sending each vertex of $G$ to the unique vertex of $R_n$, and each open edge of $G$ homeomorphically to one edge of $R_n$.
\end{mydef}

This means that every edge of $G$ is equipped with a label in $\{x_1,...,x_n\}$ and an orientation, according to which edge of $R_n$ it is mapped to; the map $f:G\rar R_n$ is called \textbf{labeling map} for $G$.

\begin{mydef}
Let $G_0$ and $G_1$ be $\grafos$. A map $h:G_0\rar G_1$ is called \textbf{label-preserving} if it sends each vertex to a vertex and each edge to an edge with the same label and orientation.
\end{mydef}

We notice that, if $h:G_0\rar G_1$ is label-preserving and $f_0:G_0\rar R_n$ and $f_1:G_1\rar R_n$ are the labeling maps, then we have $f_1\circ h=f_0$.

\subsection*{Core graph of a subgroup}

\begin{mydef}
Let $G$ be a graph which is not a tree. Define its \textbf{core graph} $\core{G}$ as the subgraph given by the union of its non-degenerate loops, i.e. all the images $f(S^1)$ for $f:S^1\rar G$ a continuous locally injective map from the circle.
\end{mydef}

Notice that $\core{G}$ is connected, and every vertex has valence at least $2$. We say that a graph $G$ is \textbf{core} if $\core{G}=G$.

In the following we will often consider graphs with a basepoint. We always mean that the basepoint is a vertex of the graph.

\begin{mydef}
Let $(G,*)$ be a pointed graph which is not a tree. Define its \textbf{pointed core graph} $\bcore{G}$ as the subgraph given by the union of all the images $f([0,1])$ for $f:[0,1]\rar G$ a continuous locally injective map with $f(0)=f(1)=*$.
\end{mydef}

We will oftet abbreviate $(G,*)$ to $G$. For a pointed graph $G$, there is a unique shortest path $\sigma$ (either trivial or embedded) connecting the basepoint to $\core{G}$; the graph $\bcore{G}$ consists exactly of the union $\core{G}\cup\im{\sigma}$.

Given a nontrivial subgroup $H\sgr F_n$, we can build the corresponding pointed covering space $(\cov{H},*)\rar(R_n,*)$. Define the core graph $\core{H}$ and the pointed core graph $\bcore{H}$ to be the core and the pointed core of $(\cov{H},*)$, respectively. Notice that the labeling map $f:\bcore{H}\rar R_n$ induces an injective map $\pi_1(f):\pi_1(\bcore{H},*)\rar F_n$, and the image of such map is exactly the subgroup $H$. We observe that conjugate subgroups have the same core graph, but distinct pointed core graphs. We also observe that $H$ is finitely generated if and only if $\core{H}$ is finite (and if and only if $\bcore{H}$ is finite).

\subsection*{Stallings' folding}

We will assume that the reader has some confidence with the classical Stallings' folding operation, for which I refer to \cite{Stallings}. I briefly recall the main properties that we are going to use.

Let $G$ be a finite connected $\grafo$ and suppose there are two distinct edges $e_1,e_2$ with endpoints $v,v_1$ and $v,v_2$ respectively. Suppose that $e_1$ and $e_2$ have the same label and orientation. We can identify $v_1$ with $v_2$, and $e_1$ with $e_2$: we then get a quotient map of graphs $q:G\rar G'$.

\begin{mydef}
The quotient map $q:G\rar G'$ is called \textbf{Stallings' folding}.
\end{mydef}

We notice that $q$ is label preserving. Fix a basepoint $*\in G$, which induces a basepoint $*\in G'$: then the map $\pi_1(f):\pi_1(G)\rar\pi_1(R_n)$ and the map $\pi_1(f'):\pi_1(G')\rar\pi_1(R_n)$ give the same subgroup $\pi_1(f)(\pi_1(G))=\pi_1(f')(\pi_1(G'))\sgr\pi_1(R_n)=F_n$. The map $\pi_1(q):\pi_1(G)\rar\pi_1(G')$ is surjective; however, it is not injective in general.

\begin{mydef}
A Stallings' folding $q:G\rar G'$ is called \textbf{rank-preserving} if the induced map $\pi_1(q)$ is an isomorphism.
\end{mydef}

Being rank-preserving is equivalent to the requirement $v_1\not=v_2$ (i.e. that we are identifying two distinct vertices) (see also figure \ref{rankpreserving}). This does not depend on the choice of the basepoint.

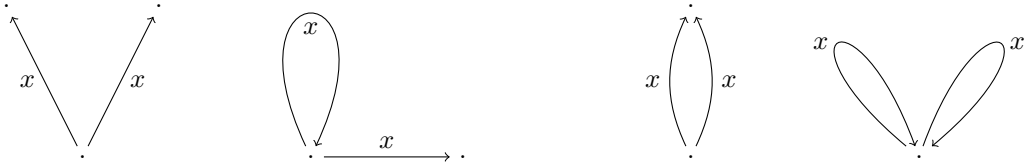
\begin{figure}[h!]
\centering
\begin{tikzpicture}
\node (1) at (2,1) {.};
\node (2) at (1,3) {.};
\node (3) at (3,3) {.};
\draw[->] (1) to node[left]{$x$} (2);
\draw[->] (1) to node[right]{$x$} (3);

\node (4) at (5,1) {.};
\node (5) at (7,1) {.};
\draw[->] (4) to[out=115,in=65,looseness=50] node[below]{$x$} (4);
\draw[->] (4) to node[above]{$x$} (5);

\node (6) at (10,1) {.};
\node (7) at (10,3) {.};
\draw[->] (6) to[out=115,in=-115,looseness=1] node[left]{$x$} (7);
\draw[->] (6) to[out=65,in=-65,looseness=1] node[right]{$x$} (7);

\node (8) at (13,1) {.};
\draw[->] (8) to[out=140,in=110,looseness=50] node[left]{$x$} (8);
\draw[->] (8) to[out=70,in=40,looseness=50] node[right]{$x$} (8);
\end{tikzpicture}
\caption{Examples of configurations where a folding operation is possible. The two examples on the left produce rank-preserving folding operations; the two examples on the right produce non-rank-preserving folding operations.}
\label{rankpreserving}
\end{figure}

Given a finite $\grafo$ $G$, we can successively apply folding operations to $G$ in order to get a sequence $G=G^{(0)}\rar G^{(1)}\rar...\rar G^{(l)}$. Notice that the number of edges decreases by $1$ at each step, and thus the length of any such chain is bounded (by the number of the edges of $G$).

\begin{myprop}\label{folding}
Let $G$ be a finite connected $\grafo$ and let $G=G^{(0)}\rar G^{(1)}\rar...\rar G^{(l)}$ be a maximal sequence of folding operations. Also, fix a basepoint $*\in G$, inducing a basepoint $*\in G^{(i)}$. Then we have the following:

(i) Each such sequence has the same length $l$ and the same final graph $G^{(l)}$.

(ii) For each such sequence, and for each label, the sequence has the same number of folding operations involving edges with that label.

(iii) Each such sequence has the same number of rank-preserving folding operations.

(iv) Let $f^{(i)}:G^{(i)}\rar R_n$ be the labeling map. Then the image of $\pi_1(f^{(i)}):\pi_1(G^{(i)})\rar\pi_1(R_n)$ is the same subgroup $H\sgr F_n$ for every $i$.

(v) For every $i$, there is a unique label-preserving map of pointed graphs $h^{(i)}:G^{(i)}\rar\cov{H}$. The image $\im{h^{(i)}}$ is the same subgraph of $\cov{H}$ for every $i$.

(vi) The map $h^{(l)}$ is an embedding of $G^{(l)}$ as a subgraph of $\cov{H}$. Moreover, $G^{(l)}$ contains $\bcore{H}$.
\end{myprop}

An $\grafo$ is called \textbf{folded} if no folding operation is possible on $G$.

\begin{mydef}
Let $G$ be a finite connected $\grafo$. Define its \textbf{folded graph} $\fold{G}$ to be the $\grafo$ $G^{(l)}$ obtained from any maximal sequence of folding operations as in proposition \ref{folding}.
\end{mydef}

A combinatorial cycle inside $G$ is called \textbf{reduced} if it never crosses consecutively two edges with the same label but opposite orientation.

\begin{mylemma}\label{novalence1}
Let $G$ be a finite connected $\grafo$ and let $G\rar G'\rar...\rar G^{(l)}$ be any maximal sequence of folding operations as in proposition \ref{folding}. Suppose that every vertex of $G$ is contained in some reduced cycle. Then there is no valence-$1$ vertex in any graph of the sequence.
\end{mylemma}

Suppose we are given a finite set of reduced words $w_1,...,w_k\in F_n$ of lengths $l_1,...,l_k$, and let $H=\gen{w_1,...,w_k}$. We can construct the graph $G$ given by a basepoint $*$ and pairwise disjoint loops $\gamma_1,...,\gamma_k$ starting and ending at the basepoint. The loop $\gamma_i$ is subdivided into $l_i$ edges, labeled and oriented according to the letters of the word $w_i$, in such a way that, when going along $\gamma_i$, we read exactly the word $w_i$. We have the following proposition:

\begin{mylemma}\label{constructioncore}
We have that $\bcore{H}=\fold{G}$. Moreover, along the chain of folding operations, we never have a valence-$1$ vertex, except possibly for the basepoint.
\end{mylemma}

In particular, $w_1,...,w_n$ are a basis for $F_n$ if and only if, with the above construction, the resulting $\grafo$ $\fold{G}$ is the standard $n$-rose $R_n$. In that case, each folding operation has to be rank-preserving, since the fundamental group of $G$ has the same rank as the fundamental group of $\fold{G}=R_n$.

\subsection*{Primitive elements and free factors}

We will be interested in the study of primitive elements and free factors:

\begin{mydef}
An element $w\in F$ is called \textbf{primitive} if it is part of some basis for the group.
\end{mydef}

\begin{mydef}
A subgroup $H\sgr F$ is called a \textbf{free factor} if some basis(equivalent: every basis) for $H$ can be extended to a basis for $F$.
\end{mydef}

Notice that an element $w\in F$ is primitive if and only if the cyclic subgroup $\gen{w}$ is a free factor; in this sense, the notion of free factor is a natural generalization of the notion of primitive element.

The following lemma is immediate:

\begin{mylemma}
Let $G$ be a pointed graph and let $G'$ be a connected subgraph containing the basepoint. Then the map $\pi_1(G')\rar\pi_1(G)$ induced by the inclusion is injective, and $\pi_1(G')$ is a free factor in $\pi_1(G)$.
\end{mylemma}

The following proposition turns out to be very useful in several situations.

\begin{myprop}
Let $K\sgr F_n$ be a finitely generated subgroup, and let $H\sgr F_n$ be a free factor. Then $H\cap K$ is a free factor in $K$.
\end{myprop}

\begin{proof}
Suppose $H$ has rank $r$: without loss of generality, we can assume that $H=\gen{x_1,...,x_r}\sgr F_n$.

Let $G=\bcore{K}$ be the basepointed $\grafo$ which represents $K$. A word $w$ belongs to $H\cap K$ if and only if it can be represented by a path inside $G$ which starts and ends at the basepoint, and which only crosses edges labeled with $x_1,...,x_r$. Consider the subgraph $G'\subseteq G$ which is given by the union of the basepoint and of all such paths. Then $\pi_1(G')$ is exactly $H\cap K$.

But since $G'$ is a subgraph of $G$, we have that $\pi_1(G')$ is a free factor in $\pi_1(G)$, meaning that $H\cap K$ is a free factor in $K$, as desired.
\end{proof}

We can easily obtain several corollaries from the above proposition.

\begin{mycor}\label{primitiveomitambient}
For $k<n$, consider the standard inclusion $F_k=\gen{x_1,...,x_k}\sgr\gen{x_1,...,x_n}=F_n$. Let $w\in F_k$. Then $w$ is primitive in $F_k$ if and only if $w$ is primitive in $F_n$.
\end{mycor}

\begin{mycor}\label{freefactoromitambient}
For $k<n$, consider the standard inclusion $F_k=\gen{x_1,...,x_k}\sgr\gen{x_1,...,x_n}=F_n$. Let $H\sgr F_k$. Then $H$ is a free factor in $F_k$ if and only if $H$ is a free factor in $F_n$.
\end{mycor}

In particular, when talking about primitive elements and free factors, it often makes sense to omit mention of the ambient group.

\begin{mycor}\label{intersection}
Let $H,H'\sgr F_n$ be free factors. Then $H\cap H'$ is a free factor.
\end{mycor}

In the above statement, we mean that it is a free factor in $F_n$, and also in both $H$ and $H'$.

\section{A fine property of Whitehead's algorithm}\label{WhiteheadWords}

\subsection*{Whitehead automorphisms and Whitehead graph}

\begin{mydef}
Let $a\in\{x_1,...,x_n,\ol{x_1},...,\ol{x_n}\}$ and let $A\subseteq\{x_1,...,x_n,\ol{x_1},...,\ol{x_n}\}\setminus\{a,\ol a\}$. Define the \textbf{Whitehead automorphism} $\varphi=(A,a)$ as the automorphism given by $a\mapsto a$ and

\begin{center}
$\begin{cases}
x_j\mapsto x_j & \text{if } x_j,\ol{x_j}\not\in A\\
x_j\mapsto ax_j & \text{if } x_j\in A \text{ and } \ol{x_j}\not\in A\\
x_j\mapsto x_j\ol{a} & \text{if } x_j\not\in A \text{ and } \ol{x_j}\in A\\
x_j\mapsto ax_j\ol{a} & \text{if } x_j,\ol{x_j}\in A\\
\end{cases}$
\end{center}
\end{mydef}

The letter $a$ will be called the \textbf{acting letter}, and the set $A$ will be the set of letters we \textbf{act on}. Notice that my notation for Whitehead automorphisms is slightly different from the one found in Lyndon and Schupp's book \cite{LyndonSchupp}: they choose to include the acting letter $a$ inside the set $A$, while I prefer not to do so.

Let $w$ be a cyclically reduced word, whose reduced form is $w=b_1...b_l$, where we have $b_j\in\{x_1,...,x_n,\ol{x_1},...,\ol{x_n}\}$. Let $\varphi=(A,a)$ be a Whitehead automorphism. We can substitute each $b_j$ with the sequence $\varphi(b_j)$ (which is either $b_j$ or $ab_j$ or $b_j\ol{a}$ or $ab_j\ol{a}$): this produces a new writing $\varphi(b_1)...\varphi(b_l)$ which represents the word $\varphi(w)$; this writing will not be cyclically reduced in general. In what follows, when we speak of the free or cyclic reduction, we mean any sequence of moves in which an adjacent pair of letters $x_i\ol{x_i}$ or $\ol{x_i}x_i$ is replaced by the empty word.

\begin{mylemma}\label{onlya}
Let $w=b_1...b_l$ be a cyclically reduced word, and let $\varphi=(A,a)$ be a Whitehead automorphism. Then, in the process of cyclic reduction for the sequence $\varphi(b_1)...\varphi(b_l)$, no letter different from $a$ gets cancelled.
\end{mylemma}
\begin{proof}
For simplicity of notation, we prove the proposition only for the free reduction process; the proof for the cyclic reduction process is completely analogous.

Fix a process of free reduction for $\varphi(b_1)...\varphi(b_l)$. Suppose some cancellation takes place, involving a letter which is not $a$ nor $\ol{a}$. Consider the first such cancellation. Suppose this cancellation involves a letter from the block $\varphi(b_j)$ and one from the block $\varphi(b_k)$, with $j<k$. Then we must have $b_k=\ol{b_j}$, and either all the letters inbetween are $a$, or all of them are $\ol{a}$. This means that the word $w$ has the form either $...b_ja^d\ol{b_j}...$ or $...b_j\ol{a}^d\ol{b_j}...$ for some $d\ge0$. Also, since the writing $b_1...b_l$ was reduced, we must have $d>0$.

We assume $w$ has the form $...b_ja^d\ol{b_j}...$, the other case being completely analogous. If $\ol{b_j}\not\in A$, then the sequence $\varphi(b_1)...\varphi(b_l)$ has the form $...b_ja^d\ol{b_j}...$, and at least one $a$ letter survives between $b_j$ and $\ol{b_j}$, and thus $b_j$ is not allowed to cancel with $\ol{b_j}$. If $\ol{b_j}\in A$, then the sequence $\varphi(b_1)...\varphi(b_l)$ has the form $...b_j\ol{a}a^da\ol{b_j}...$, and again we see that at least one $a$ letter survives between $b_j$ and $\ol{b_j}$, and thus $b_j$ is not allowed to cancel with $\ol{b_j}$. This contradiction completes the proof.
\end{proof}

\begin{mydef}\label{defWhiteheadgraph}
Let $w$ be a cyclically reduced word. Define the \textbf{Whitehead graph} of $w$ as follows:

(i) We have $2n$ vertices labeled $x_1,...,x_n,\ol{x_1},...,\ol{x_n}$.

(ii) For every pair of consecutive letters in $w$, we draw an arc from the inverse of the first letter to the second. We also draw an arc connecting the inverse of the last letter of $w$ to the first letter of $w$, as if they were adjacent.
\end{mydef}

Notice that, $w$ being cyclically reduced, we never have any arc connecting a vertex to itself.

\begin{mydef}\label{defcutvertex}
Let $w$ be a cyclically reduced word. A vertex $a$ in the Whitehead graph of $w$ is called a \textbf{cut vertex} if it is non-isolated and at least one of the following two configurations happens:

(i) The connected component of $a$ doesn't contain $\ol a$.

(ii) The connected component of $a$ becomes disconnected if we remove $a$.
\end{mydef}

\subsection*{Whitehead's algorithm}

We are now ready to state Whitehead's theorem and our refinement of it (theorem \ref{finesse}).

\begin{mythm}\label{cutvertex}
Let $w$ be a cyclically reduced word, which is primitive but not a single letter. Then the Whitehead graph of $w$ contains a cut vertex.
\end{mythm}

\begin{mythm}\label{Whitehead}
Let $w$ be a cyclically reduced word, and suppose the Whitehead graph of $w$ contains a cut vertex. Then there is a Whitehead automorphism $\varphi$ such that the cyclic length of $\varphi(w)$ is strictly smaller than the cyclic length of $w$.
\end{mythm}

\begin{mythm}\label{finesse}
The automorphism in theorem \ref{Whitehead} can be chosen in such a way that every $a$ or $\ol a$ letter, which is added when we apply $\varphi$ to $w$ letter by letter, immediately cancels (in the cyclic reduction process).
\end{mythm}

\begin{myex}
Let $w=xyxyx\ol yz$ in $F_3$. Consider the automorphism $\varphi=(\{\ol x\},y)$, meaning that $x\mapsto x\ol y$ and $y\mapsto y$ and $z\mapsto z$: the word becomes $\varphi(w)=xxx\ol y\ol yz$. This is shorter, so is would be fine for theorem \ref{Whitehead}. But it is not fine for theorem \ref{finesse}, because a letter $\ol y$ appears between the last $x$ and the $z$, and it does not cancel.
\end{myex}

We are now going to prove theorems \ref{cutvertex}, \ref{Whitehead} and \ref{finesse}. The proof of \ref{cutvertex} that we give is essentially contained in \cite{Heusener}, but we will make use of variations of the argument in what follows, so I prefer to rewrite it here.

\begin{proof}[Proof of theorem \ref{cutvertex}]
Let $w$ be cyclically reduced and primitive. We will assume that $w$ contains all the letters $x_1,...,x_n$ at least once: otherwise, if $w$ only contains the letters $x_1,...,x_k$, then we can just apply the same argument in the free factor $\gen{x_1,...,x_k}\sgr\gen{x_1,...,x_n}=F_n$ (using corollary \ref{primitiveomitambient}).

Since $w$ is primitive, we can take a basis $w=w_1,w_2,...,w_n$ of reduced words. We can build the graph $G$ given by a basepoint $*$, together with a path $p_i$ for each $w_i$: the path $p_i$ goes from $*$ to $*$, and contains an edge for each letter appearing in $w_i$ (in such a way that, moving around the path $p_i$, we read exactly the word $w_i$). Let $G(w)$ denote the subgraph of $G$ given by the only cycle corresponding to the generator $w$.

We now apply a sequence of folding operations to the graph $G$, in order to get a sequence $G\rar G'\rar...\rar G^{(l-1)}\rar G^{(l)}$ as in proposition \ref{folding}: each map $G^{(i)}\rar G^{(i+1)}$ consists of a single folding operation, and no further folding operation can be applyed to $G^{(l)}$. Since $w_1,...,w_n$ is a basis, we have that $G^{(l)}$ is the standard $n$-rose $R_n$. Since $w_1$ is cyclically reduced and $w_2,...,w_n$ are reduced, lemma \ref{novalence1} yields that no graph $G^{(i)}$ contains any valence-$1$ vertex. A folding operation can decrease the rank of the fundamental group, but it can't increase it; since $\pi_1(G)$ has the same rank as $\pi_1(R_n)$, we must have that each folding operation is rank-preserving.

We now look at the graph $G^{(l-1)}$: it doesn't contain any valence-$1$ vertex, and a single rank-preserving folding operation sends it to the standard $n$-rose. It is quite easy to see that $G^{(l-1)}$ has to be of the form described in figure \ref{prerosa} below, for some $1\le\alpha\le\beta\le n$ with $\alpha<n$ (up to permutation of the letters, and up to substitution of some letter with its inverse).

We have a map of graphs $f:G(w)\rar G^{(l-1)}$ preserving the orientations and the labels on the edges (given by the inclusion $G(w)\rar G$ followed by the sequence of foldings). Suppose we have two adjacent letters $w=...yz...$: this means that $G(w)$ contains two edges labeled $y$ and $z$ with a common endpoint $u$. We either have $f(u)=v$ or $f(u)=v'$, meaning that $\ol{y}$ and $z$ are either both in $\{x_1,\ol{x_1},x_2,\ol{x_2},...,x_\alpha,\ol{x_\alpha},x_{\alpha+1},...,x_\beta\}$ or both in $\{\ol{x_1},\ol{x_{\alpha+1}},...,\ol{x_\beta},x_{\beta+1},\ol{x_{\beta+1}},...,x_n,\ol{x_n}\}$. This tells us that, if we remove the vertex $\ol{x_1}$ from the Whitehead graph of $w$, we get the disjoint union $V\sqcup V'$ of two separate graphs: $V$ with vertices $x_1,x_2,\ol{x_2},...,x_\alpha,\ol{x_\alpha},x_{\alpha+1},...,x_\beta$ and $V'$ with vertices $\ol{x_{\alpha+1}},...,\ol{x_\beta},x_{\beta+1},\ol{x_{\beta+1}},...,x_n,\ol{x_n}$.

If the image $f(G(w))\subseteq G^{(l-1)}$ crosses both the edges labeled $x_1$, then in the Whitehead graph of $w$ we have that $\ol{x_1}$ is connected to at least one vertex in $V$ and to one vertex in $V'$; this means $\ol{x_1}$ is a cut vertex. If $f(G(w))$ crosses the edge labeled $x_1$ with distinct endpoints, but not the other, then $\ol{x_1}$ is connected to $V'$ but not to $V$; and again $\ol{x_1}$ is a cut vertex. If $f(G(w))$ does not cross the arc labeled $x_1$ with distinct endpoints, then we make use of the assumption that $w$ contains every letter at least once; we get that $f(G(w))$ has to contain the edge $x_{\alpha+i}$ for some $1\le i\le\beta$; this gives that any of $x_{\alpha+i},\ol{x_{\alpha+i}}$ is a cut vertex for the Whitehead graph of $w$.
\end{proof}

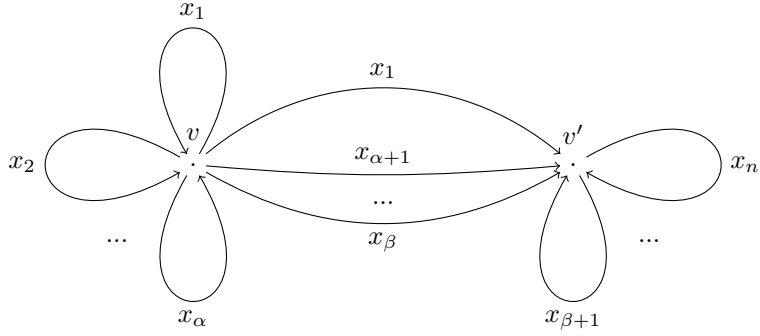
\begin{figure}[h!]
\centering
\begin{tikzpicture}
\node (1) at (0,0) {.};
\node (2) at (5,0) {.};
\node () at (0,0.4) {$v$};
\node () at (5,0.4) {$v'$};

\draw[->] (1) to[out=60,in=120,looseness=40] node[above]{$x_1$} (1);
\draw[->] (1) to[out=150,in=210,looseness=35] node[left]{$x_2$} (1);
\node (3) at (-1,-1) {...};
\draw[->] (1) to[out=240,in=300,looseness=40] node[below]{$x_\alpha$} (1);

\draw[->] (1) to[out=40,in=140,looseness=1] node[above]{$x_1$} (2);
\draw[->] (1) to[out=-5,in=185,looseness=1] node[above]{$x_{\alpha+1}$} (2);
\node (4) at (2.5,-0.5) {...};
\draw[->] (1) to[out=-30,in=210,looseness=1] node[below]{$x_\beta$} (2);

\draw[->] (2) to[out=30,in=-30,looseness=35] node[right]{$x_n$} (2);
\node (5) at (6,-1) {...};
\draw[->] (2) to[out=-60,in=-120,looseness=40] node[below]{$x_{\beta+1}$} (2);
\end{tikzpicture}
\caption{The generic graph $G^{(l-1)}$. This contains exactly one edge with each label, except for the two edges labeled $x_1$. Those two edges have to be folded in order to obtain the $n$-rose.}\label{prerosa}
\end{figure}

We now prove that theorem \ref{cutvertex} implies theorem \ref{Whitehead}, together with the fine property of theorem \ref{finesse}.

\begin{proof}[Proof of theorems \ref{Whitehead} and \ref{finesse}]
Let $w$ be cyclically reduced, and let $a$ be a cut vertex in its Whitehead graph.

If the connected component of $a$ does not contain $\ol{a}$, then we take the set $A$ to be that connected component (excluding $a$ itself). Otherwise, take the connected component of $a$ and remove $a$ itself: we are left with at least two nonempty connected components, and at least one of these components does not contain $\ol a$; take $A$ to be such a component. In both cases we consider the Whitehead automorphism $\varphi=(A,a)$. We look at what happens between two consecutive non-$a$ non-$\ol a$ letters in $w$ when we apply the automorphism $\varphi$.

Suppose we have two consecutive letters $w=...yz...$ with $y,z\in\{x_1,...,x_n,\ol{x_1},...,\ol{x_n}\}\setminus\{a,\ol a\}$. Then we have an arc from $\ol{y}$ to $z$ in the Whitehead graph, so we are either acting on neither of $\ol{y},z$ or on both of them. If we are not acting on them, then $\varphi(w)=...yz...$ and the word is not affected between $y$ and $z$. If we are acting on both of them, then $\varphi(w)=...y\ol aaz...=...yz...$ and every $a$ and $\ol a$ which appears there immediately cancels, and again the word is not affected between $y$ and $z$.

Suppose we have in $w$ a segment of the form $w=...ya^kz...$ with $y,z\in\{x_1,...,x_n,\ol{x_1},...,\ol{x_n}\}\setminus\{a,\ol a\}$ and $k\ge1$ (the case $k\le-1$ is analogous). This means that we have an arc from $\ol{a}$ to $z$, and thus we are not acting on $z$.
If we are not acting on $\ol y$, then $\varphi(w)=...ya^kz...$ and the word is not affected between $y$ and $z$.
If we are acting on $\ol y$, then $\varphi(w)=...y\ol aa^kz...=...ya^{k-1}z...$ and the number of $a$ letters strictly decreases between $y$ and $z$.

This shows that the finesse \ref{finesse} holds for this automorphism. To conclude we notice that, since there is at least one arc between $a$ and a vertex of $A$, at least one cancellation takes place, giving a (strict) decrease in the cyclic length of $w$, yielding theorem \ref{Whitehead}.
\end{proof}

\section{About primitive elements in a subgroup of a free group}\label{InterestingCriterion}

Let $H\sgr F_n$ be any subgroup. If $H$ is a free factor, then of course every element which is primitive in $H$ has to be primitive in $F_n$ too. We here deal with a converse: if $H$ is not a free factor, then there is an element which is primitive in $H$, but not in $F_n$. This section is completely dedicated to the proof of the existence of such a witness.

\begin{mythm}\label{primitiveinfreefactors}
Let $H\sgr F_n$ be a finitely generated subgroup. Suppose that every element of $H$ which is primitive in $H$ is also primitive in $F_n$. Then $H$ is a free factor.
\end{mythm}

\begin{proof}
The proof proceeds by induction on the rank of the subgroup. For the base step, we notice that for a subgroup of rank $1$ the statement is trivially true. For the inductive step, suppose we know the statement to be true for subgroups of rank $k$, and we want to prove it for subgroups of rank $k+1$.

Take a subgroup $H=\gen{w_1,...,w_{k+1}}$ of rank $k+1$ (meaning that $w_1,...,w_{k+1}$ is a basis for $H$), and suppose that every element $v$ which is primitive in $H$ is also primitive in $F$. We consider the subgroup $H'=\gen{w_1,...,w_k}$, and we notice that every element $v$ which is primitive in $H'$ is also primitive in $H$, and thus is primitive in $F$. Then $H'$ has rank $k$ and satisfies the hypothesis of the theorem, so by inductive hypothesis we get that $H'$ is primitive. So we can take an automorphism $\theta:F\rar F$ with $\theta(w_i)=x_i$ for $i=1,...,k$. Instead of proving that $H$ primitive in $F$, we prove that $\theta(H)$ is primitive in $F$; but $\theta(H)=\gen{x_1,...,x_k,\theta(w_{k+1})}$, so it is enough to prove the statement of theorem \ref{primitiveinfreefactors} for subgroups $H$ of the form $H=\gen{w,x_1,...,x_k}$.

The statement in the case of subgroups of the form $H=\gen{w,x_1,...,x_k}$ will be proved by induction on the length of $w$. The base step where $w$ has length one is trivial.

We observe that, if the first (or the last) letter of $w$ is $y\in\{x_1,...,x_k,\ol{x_1},...,\ol{x_k}\}$, then we can define the shorter word $w'=\ol{y}w$ (or $w\ol{y}$). But then $H=\gen{w,x_1,...,x_k}=\gen{w',x_1,...,x_k}$, and so we are done by inductive hypothesis. Thus in the following we assume this is not the case.

We consider the word $v=x_1wx_1z$ where $z$ is any word in the letters $\{x_1,...,x_k\}$ with the following properties:

(i) The Whitehead graph of $z$ contains at least one edge joining each pair of distinct vertices in $\{x_1,...,x_k,\ol x_1,...,\ol x_k\}$.

(ii) When we write $x_1wx_1z$ we get a cyclically reduced word, without any cancellation needed.

\noindent For example we may take $z=(x_1x_1)(x_2x_2)...(x_kx_k)\prod_{1\le i<j\le k}x_1(x_ix_j)(x_i\ol{x_j})$, where the factors in the product are ordered lexicographically (but any ordering works).

The word $v$ is primitive in the subgroup $H$, so by the hypothesis it has to be primitive in $F$, and in particular we can take an automorphism $\varphi$ satisfying theorems \ref{Whitehead} and \ref{finesse}. We now look at what happens to the letters $x_1,...,x_k$ and to the word $w$.

\

CASE 1: Suppose the acting letter $a$ is different from $x_1,...,x_k,\ol x_1,...,\ol x_k$.

Then $\{x_1,...,x_k,\ol x_1,...,\ol x_k\}$ is either contained in $A$ or disjoint from $A$. This is because the vertices $\{x_1,...,x_k,\ol x_1,...,\ol x_k\}$ are all pairwise connected in the Whitehead graph of $v$.

SUBCASE 1.1: Suppose $\{x_1,...,x_k,\ol x_1,...,\ol x_k\}$ is disjoint from $A$.

This means that $\varphi(x_1)=x_1,...,\varphi(x_k)=x_k$ and $\varphi\gen{w,x_1,...,x_k}=\gen{\varphi(w),x_1,...,x_k}$.

We have $\varphi(v)=x_1\varphi(w)x_1z$ which is cyclically reduced, and thus has to be strictly shorter than $v=x_1wx_1z$. But in $x_1\varphi(w)x_1z$ we can only have cancellations inside $\varphi(w)$, and so we are able to deduce that $\varphi(w)$ is strictly shorter than $w$, and we are done by inductive hypothesis.

SUBCASE 1.2: $\{x_1,...,x_k,\ol x_1,...,\ol x_k\}$ is contained in $A$.

This means that $\varphi(x_1)=\ol ax_1a,...,\varphi(x_k)=\ol ax_ka$.

Moreover we have $\varphi\gen{w,x_1,...,x_k}=\ol a\gen{a\varphi(w)\ol a,x_1,...,x_k}a$.

We have $\varphi(v)=\ol a x_1(a\varphi(w)\ol a)x_1za$ which cyclically reduces to $x_1(a\varphi(w)\ol a)x_1z$, and thus this has to be strictly shorter than $v=x_1wx_1z$. But in $x_1(a\varphi(w)\ol a)x_1z$, we can only have cancellations inside $(a\varphi(w)\ol a)$, and so we deduce that $a\varphi(w)\ol a$ is strictly shorter than $w$, and we are done by inductive hypothesis.

\

CASE 2: Suppose the acting letter $a$ is one of $x_2,...,x_k,\ol x_2,...,\ol x_k$.

Then $\{x_1,...,x_k,\ol x_1,...,\ol x_k\}$ is disjoint from $A$, because all the vertices $\{x_1,...,x_k,\ol x_1,...,\ol x_k\}\setminus\{a,\ol a\}$ are connected to $\ol a$ in the Whitehead graph of $v$.

Now we have $\varphi(x_1)=x_1,...,\varphi(x_k)=x_k$ and $\varphi\gen{w,x_1,...,x_k}=\gen{\varphi(w),x_1,...,x_k}$.

We proceed exactly as in subcase 1.1. The key point is that, since $a\not=x_1,\ol x_1$, when we write $x_1\varphi(w)x_1z$ the two $x_1$ letters cannot cancel against $\varphi(w)$. We get that $\varphi(w)$ is strictly shorter than $w$, and we are again done by inductive hypothesis.

\

CASE 3: Suppose $a=x_1$ (the case $a=\ol x_1$ is completely analogous).

As in case 2, we get that $\{x_1,...,x_k,\ol x_1,...,\ol x_k\}$ is disjoint from $A$, and thus $\varphi(x_1)=x_1,...,\varphi(x_k)=x_k$ and $\varphi\gen{w,x_1,...,x_k}=\gen{\varphi(w),x_1,...,x_k}$.

We have $\varphi(v)=(x_1\varphi(w))x_1z$ which is cyclically reduced, and has thus to be strictly shorter than $v=x_1wx_1z$. In $(x_1\varphi(w))x_1z$, the only cancellations can happen inside $(x_1\varphi(w))$, so we are able to deduce that $x_1\varphi(w)$ has to be strictly shorter than $x_1w$. Now some care is needed.

If $x_1$ does not cancel against $\varphi(w)$, then we get that $\varphi(w)$ is strictly shorter than $w$ and we are done by inductive hypothesis.

Otherwise, let $t$ be the first letter of $w$, and we must have $t\in A$. We look at the arcs between $a=x_1$ and $A$ in the Whitehead graph of $v$: we certainly have at least one arc from $x_1$ to $t$.

SUBCASE 3.1: If we have more than one arc from $x_1$ to $t$, or if we have any other arc from $x_1$ to $A$, then all of those arcs give cancellations inside $\varphi(w)$. We already knew that $x_1\varphi(w)$ was strictly shorter than $x_1w$, but now we got at least one additional cancellation inside $\varphi(w)$, so we are able to deduce that $x_1\varphi(w)$ is strictly shorter than $w$. Now we observe that $\varphi\gen{w,x_1,...,x_k}=\gen{x_1\varphi(w),x_1,...,x_k}$, and we are done by inductive hypothesis.

SUBCASE 3.2: Suppose we only have one arc from $x_1$ to $t$, and no other arc from $x_1$ to $A$. If the vertex $t$ has degree at least $2$, then we use $t$ instead of $x_1$ as cut vertex for the Whitehead graph of $v$, and we end up in case 1, and we are done. If the vertex $t$ has degree $1$, then the letter $t$ appears exactly once inside $w$; in this case we have the automorphism $\eta:F\rar F$ which keeps all the letters fixed, except for $t\mapsto w$; this gives $\eta\gen{t,x_1,...,x_k}=\gen{w,x_1,...,x_k}$, showing that the subgroup is primitive, as desired.
\end{proof}

\begin{myrmk}
We proved the theorem for $H$ finitely generated subgroup, but that hypothesis can easily be waived. For a subgroup of rank greater than the rank of $F$ (and, as a consequence, for a subgroup of infinite rank), the hypothesis of theorem \ref{primitiveinfreefactors} can't hold.
\end{myrmk}

\section{Whitehead's algorithm for free factors}\label{WhiteheadSubgroups}

\subsection*{Whitehead graph for subgroups}

\begin{mydef}\label{Letters}
Let $G$ be a $\grafo$ and let $v\in G$ be a vertex. Define the \textbf{letters at $v$} to be the subset $L(v)\subseteq\{x_1,...,x_n,\ol{x_1},...,\ol{x_n}\}$ of the labels of the edges coming out of $v$. More precisely, we have $x_i\in L(v)$ if and only if $G$ contains an edge labeled $x_i$ coming out of $v$, and $\ol{x_i}\in L(v)$ if and only if $G$ contains an edge labeled $x_i$ going into $v$.
\end{mydef}

\begin{mydef}
Let $G$ be a $\grafo$. Define the \textbf{Whitehead graph} of $G$ as follows:

(i) We have $2n$ vertices labeled $x_1,...,x_n,\ol{x_1},...,\ol{x_n}$.

(ii) For every vertex $v\in G$ and for every pair $y,z\in L(v)$ of distinct letters at $v$, we draw an arc from $y$ to $z$ in the Whitehead graph.
\end{mydef}

This means that the Whitehead graph of $G$ contains an edge for every (legal) turn in $G$. Notice that the Whitehead graph contains a complete subgraph with vertex set $L(v)$ for every vertex $v\in G$; moreover, the Whitehead graph is exactly the union of these complete subgraphs. Notice that, when we apply a folding operation to a $\grafo$ $G$, the Whitehead graph of $G$ can gain new edges, but it doesn't lose any.

We define the Whitehead graph of a nontrivial finitely generated subgroup $H\sgr F_n$ to be the Whitehead graph of $\core{H}$. When the subgroup $H$ is generated by a single word $H=\gen{w}$, the Whitehead graph of $H$ coincides with the Whitehead graph of the cyclic reduction of $w$; in this sense, this notion of Whitehead graph is a generalization of the previous definition \ref{defWhiteheadgraph}. We can also define the notion of cut vertex for the Whitehead graph of a subgroup exactly as in definition \ref{defcutvertex}.

\subsection*{Whitehead automorphisms and subdivision of graphs}

We are now interested in how the core graph of a subgroup changes when we apply a Whitehead automorphism. We are thus going to describe an operation which I call \textbf{subdivision}, which we perform on an $\grafo$. In what follows, we work with a fixed Whitehead automorphism $\varphi=(A,a)$ and we assume that $a\in\{x_1,...,x_n\}$. The case where $a\in\{\ol{x_1},...,\ol{x_n}\}$ is completely analogous: whenever we would have an edge labaled $a$ with a certain orientation, we consider instead an edge labeled $\ol{a}$ and with opposite orientation.

Let $G$ be an $\grafo$. Choose an edge $e\in G$, oriented and labeled with a letter $y\in\{x_1,...,x_n\}$. If $y\in A$ and $\ol{y}\not\in A$, then we subdivide $e$ in two edges, and to the first we give the label $a$ and the orientation of $e$, and to the second we give the label $y$ and the same orientation. If $y\not\in A$ and $\ol{y}\in A$, then we subdivide $e$ in two edges, and to the first we give the label $y$ and the same orientation of $e$, and to the second we give the label $a$ and opposite orientation. If $y,\ol{y}\not\in A$ we do not change the edge $e$, and if $y,\ol{y}\in A$ then we perform both transformations on $e$, as in figure \ref{subdivisionedges}.

\begin{figure}[h!]
\centering
\begin{tikzpicture}[scale=0.66]
\node (1) at (2,6) {.};
\node (2) at (6,6) {.};
\node (3) at (2,4) {.};
\node (4) at (6,4) {.};
\draw[->] (1) to node[above]{$x$} (2);
\draw[->] (3) to node[above]{$x$} (4);

\node (5) at (8,6) {.};
\node (6) at (12,6) {.};
\node (7) at (8,4) {.};
\node (8) at (8.8,4) {.};
\node (9) at (12,4) {.};
\draw[->] (5) to node[above]{$y$} (6);
\draw[->] (7) to node[above]{$x$} (8);
\draw[->] (8) to node[above]{$y$} (9);

\node (10) at (14,6) {.};
\node (11) at (18,6) {.};
\node (12) at (14,4) {.};
\node (13) at (17.2,4) {.};
\node (14) at (18,4) {.};
\draw[->] (10) to node[above]{$z$} (11);
\draw[->] (12) to node[above]{$z$} (13);
\draw[->] (14) to node[above]{$x$} (13);

\node (15) at (20,6) {.};
\node (16) at (24,6) {.};
\node (17) at (20,4) {.};
\node (18) at (20.8,4) {.};
\node (19) at (23.2,4) {.};
\node (20) at (24,4) {.};
\draw[->] (15) to node[above]{$t$} (16);
\draw[->] (17) to node[above]{$x$} (18);
\draw[->] (18) to node[above]{$t$} (19);
\draw[->] (20) to node[above]{$x$} (19);
\end{tikzpicture}
\caption{The effect of the subdivision operation on the single edges. Here $F_4=\gen{x,y,z,t}$ and $\varphi=(\{y,\ol{z},t,\ol{t}\},x)$, meaning that $\varphi(x)=x$ and $\varphi(y)=xy$ and $\varphi(z)=z\ol{x}$ and $\varphi(t)=xt\ol{x}$. Above we see the edges before the subdivision, while below we see them after the subdivision.}
\label{subdivisionedges}
\end{figure}
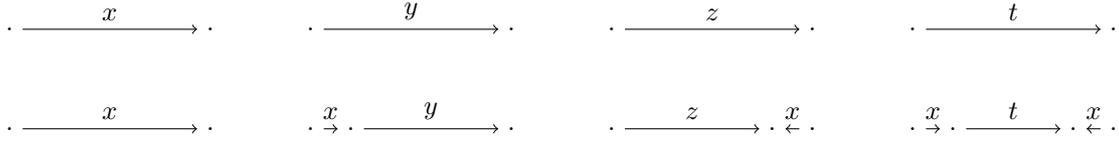

We apply the subdivision operation to each edge of $G$, in order to obtain another $\grafo$, which will be called \textbf{$\varphi$-subdivided graph}, which we denote by $\subd{\varphi}{G}$.

Notice that for every vertex $v\in G$ we also have a vertex $v\in\subd{\varphi}{G}$: we say that $v\in\subd{\varphi}{G}$ is an \textbf{old vertex}. For every edge $e\in G$, the subdivision on $e$ gives at most three edges, and exactly one of them has the same label as $e$: we say that edge is an \textbf{old edge}. If we have a vertex $v\in G$ and a letter $b\in L(v)\cap A$, then the subdivision operation on the corresponding edge will create a new vertex $u$ together with an edge labeled $a$ going from $v$ to $u$: we say that the vertex $u$ is a \textbf{new vertex near $v$}, and that the edge from $v$ to $u$ is a \textbf{new edge near $v$}. For every vertex of $\subd{\varphi}{G}$, it is either an old vertex or a new vertex near a unique $v\in G$. For every edge of $\subd{\varphi}{G}$, it is either an old edge, or a new edge near a unique $v\in G$.

To an $\grafo$ $G$ we associate the subgroup $H\sgr F_n$ given by the image $H=\pi_1(f)(\pi_1(G))$, where $f:G\rar R_n$ is the labeling map. It is immediate to see that, if $G$ is associated with the subgroup $H$, then $\subd{\varphi}{G}$ is associated with the subgroup $\varphi(H)$. This observation yields the following:

\begin{myprop}\label{subdivision}
Let $H\sgr F_n$ be a non-trivial finitely generated subgroup and let $\varphi=(A,a)$ be a Whitehead automorphism. Then $\core{\varphi(H)}=\core{\fold{\subd{\varphi}{\core{H}}}}$.
\end{myprop}

\begin{proof}
For the $\grafo$ $\core{H}$ with labeling map $f:\core{H}\rar R_n$, we have that $f_*(\pi_1(\core{H}))=H$. After the subdivision operation, if we call $g:\subd{\varphi}{\core{H}}\rar R_n$ the labeling map, we have that $g_*(\pi_1(\subd{\varphi}{\core{H}}))=\varphi(H)$. In particular, proposition \ref{folding} tells us that $\fold{\subd{\varphi}{\core{H}}}$ can be embedded in $\cov{\varphi(H)}$ as a subgraph containing $\core{\varphi(H)}$. It follows that $\core{\fold{\subd{\varphi}{\core{H}}}}=\core{\varphi(H)}$, as desired.
\end{proof}

The following technical result is a generalization of lemma \ref{onlya}.

\begin{mylemma}\label{onlya2}
Let $H\sgr F_n$ be a non-trivial finitely generated subgroup and let $\varphi=(A,a)$ be a Whitehead automorphism. Consider the graph $\subd{\varphi}{\core{H}}$. For every vertex $v\in\core{H}$, fold together all the edges of $\subd{\varphi}{\core{H}}$ going out of $v$ and labeled with $a$. Then, after these folding operations, no further folding operation is possible.

In particular, for every folding sequence starting from $\subd{\varphi}{\core{H}}$, the sequence contains only rank-preserving folding operations involving edges labeled $a$.
\end{mylemma}

\begin{figure}[h!]
\centering
\begin{tikzpicture}
\node (1) at (2,8) {.};
\node (2) at (1,11) {.};
\node (3) at (3,11) {.};
\node (4) at (5,10) {.};
\node (5) at (5,8) {.};
\node (6) at (5,6) {.};
\draw[->] (2) to node[left]{$x$} (1);
\draw[->] (3) to node[left]{$y$} (1);
\draw[->] (1) to node[above]{$y$} (4);
\draw[->] (5) to node[above]{$z$} (1);
\draw[->] (6) to node[above]{$t$} (1);
\node (7) at (2,3) {.};
\node (8) at (5,4) {.};
\node (9) at (5,2) {.};
\draw[->] (1) to node[left]{$x$} (7);
\draw[->] (7) to node[below]{$y$} (8);
\draw[->] (7) to node[below]{$z$} (9);
\node () at (1.8,7.8) {$v$};
\node () at (1.8,2.8) {$u$};
\node () at (0.5,4.5) {$G=\core{H}$};

\node (1s) at (10,8) {.};
\node (2s) at (9,11) {.};
\node (3s) at (11,11) {.};
\node (4s) at (13,10) {.};
\node (5s) at (13,8) {.};
\node (6s) at (13,6) {.};
\node (10s) at (10.66,10) {.};
\node (11s) at (11,8.66) {.};
\node (12s) at (11,8) {.};
\node (13s) at (11,7.33) {.};
\node (14s) at (12,6.66) {.};
\draw[->] (2s) to node[left]{$x$} (1s);
\draw[->] (3s) to node[left]{$x$} (10s);
\draw[->] (10s) to node[left]{$y$} (1s);
\draw[->] (1s) to node[above]{$x$} (11s);
\draw[->] (11s) to node[above]{$y$} (4s);
\draw[->] (1s) to node[above]{$x$} (12s);
\draw[->] (5s) to node[above]{$z$} (12s);
\draw[->] (1s) to node[above]{$x$} (13s);
\draw[->] (14s) to node[above]{$t$} (13s);
\draw[->] (6s) to node[above]{$x$} (14s);
\node (7s) at (10,3) {.};
\node (8s) at (13,4) {.};
\node (9s) at (13,2) {.};
\node (15s) at (11,3.33) {.};
\node (16s) at (12,2.33) {.};
\draw[->] (1s) to node[left]{$x$} (7s);
\draw[->] (7s) to node[above]{$x$} (15s);
\draw[->] (15s) to node[above]{$y$} (8s);
\draw[->] (7s) to node[above]{$z$} (16s);
\draw[->] (9s) to node[above]{$x$} (16s);
\node () at (9.8,7.8) {$v$};
\node () at (9.8,2.8) {$u$};
\node () at (8.5,4.5) {$\subd{\varphi}{G}$};
%

\node (1u) at (6,-4) {.};
\node (2u) at (5,-1) {.};
\node (3u) at (7,-1) {.};
\node (10u) at (6.66,-2) {.};
\draw[->] (2u) to node[left]{$x$} (1u);
\draw[->] (3u) to node[left]{$x$} (10u);
\draw[->] (10u) to node[left]{$y$} (1u);
\node (7u) at (6,-9) {.};
\node (8u) at (9,-8) {.};
\node (9u) at (9,-10) {.};
\node (15u) at (7,-8.66) {.};
\node (16u) at (8,-9.66) {.};
\draw[->] (1u) to node[left]{$x$} (7u);
\draw[->] (7u) to node[below]{$x$} (15u);
\draw[->] (15u) to node[below]{$y$} (8u);
\draw[->] (7u) to node[below]{$z$} (16u);
\draw[->] (9u) to node[above]{$x$} (16u);
\node (4u) at (9,-3) {.};
\node (5u) at (9,-5) {.};
\node (6u) at (9,-7) {.};
\node (14u) at (8,-7.66) {.};
\draw[->] (7u) to node[left]{$y$} (4u);
\draw[->] (5u) to node[above]{$z$} (7u);
\draw[->] (14u) to node[above]{$t$} (7u);
\draw[->] (6u) to node[above]{$x$} (14u);
\node () at (5.8,-4.2) {$v$};
\node () at (5.8,-9.2) {$u$};
\node () at (4.5,-8.5) {$G'$};
\end{tikzpicture}
\caption{A local picture of the graph $G$ and how it changes during the proof of lemma \ref{onlya} (with a focus on case 2). Here $F_4=\gen{x,y,z,t}$ and $\varphi=(\{y,\ol{z},t,\ol{t}\},x)$ is the same as in figure \ref{subdivisionedges}. We see a portion of the starting graph $G$, its subdivision, and the corresponding portion of $G'$.}\label{GG'}
\end{figure}
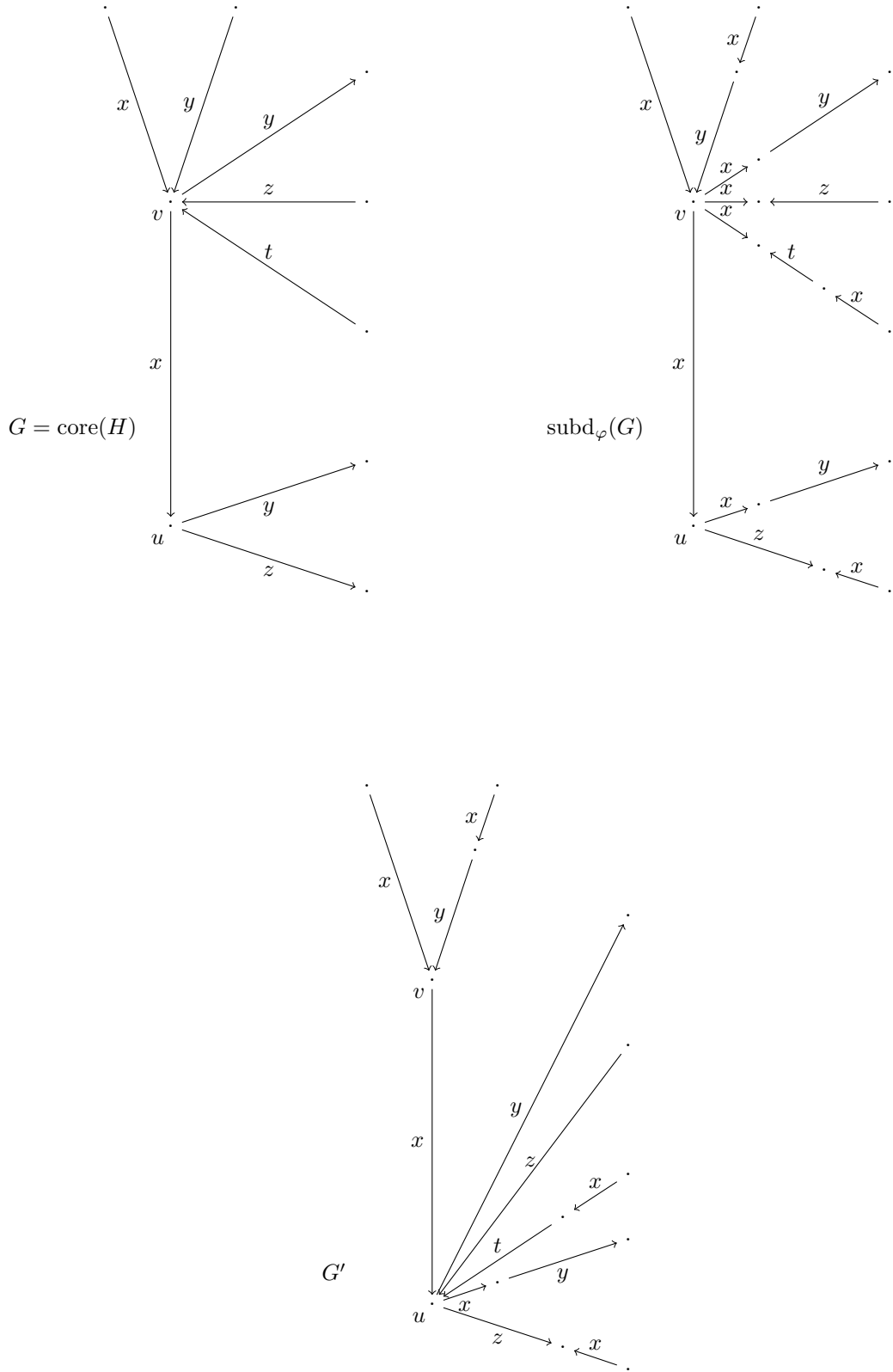

\begin{proof}
Let $G=\core{H}$ and denote with $L(v)$ the set of letters at the vertex $v$ in $G$. For every vertex $v\in G$, take all the edges in $\subd{\varphi}{G}$ labeled $a$ and going out of $v$, and fold them all together; call the resulting graph $G'$. Our aim is to show that no folding operation is possible on $G'$. This is equivalent to showing that, for every vertex $u\in G'$ and for every letter, there is at most one edge with that label going out of $u$.

Suppose we have a vertex $v\in G$ with $L(v)\cap A=\emptyset$. Then no new vertex is created near $v$ in the subdivision operation.

Suppose we have a vertex $v\in G$ with $L(v)\cap A\not=\emptyset$ and $a\in L(v)$: this means that $G$ contains an edge labeled $a$ going from $v$ to $u$. Then every new vertex, which is created near $v$ with the subdivision operation, gets identified with $u$ in $G'$. 

Suppose we have a vertex $v\in G$ with $L(v)\cap A\not=\emptyset$ and $a\not\in L(v)$. Then all the new vertices, which are created near $v$ with the subdivision operation, are identified together into a vertex which we call $v_1\in G'$.

Thus $G'$ contains only two types of vertices: old vertices $u$ obtained from a vertex $u\in G$, and new vertices $u_1$ near a vertex $u\in G$ with $L(u)\cap A\not=\emptyset$ and $a\not\in L(u)$.

CASE 1: Suppose we have a vertex $u\in G$ with $\ol{a}\not\in L(u)$. Since $\ol{a}\not\in L(u)$, we have that the vertex $u\in\subd{\varphi}{G}$ does not get identified with any other vertex during the folding operations that produce $G'$. The edges going out of $u\in G$ with label in $L(u)\cap A^c$ give edges going out of $u\in G'$ with the same label. The edges going out of $u\in G$ with label in $L(u)\cap A$ give edges labeled $a$ going out of $u\in\subd{\varphi}{G}$; all the edges labeled $a$ and going out of $u\in\subd{\varphi}{G}$ are folded together in $G'$, meaning that there is at most one edge labeled $a$ going out of $u\in G'$. Thus in this case the vertex $u\in G'$ has at most one edge with each label going out of it.

CASE 2: Suppose we have a vertex $u\in G$ with $\ol{a}\in L(u)$. This means that $G$ contains an edge labeled $a$ going from $v$ to $u$ (see also figure \ref{GG'}). It is possible that the subdivision operation creates new vertices near $v$; the folding operation identifies all such vertices with $u$. Thus, for every letter $l\in L(v)\cap A$, the vertex $u\in G'$ has one edge labeled $l$ going out of it. The edges going out of $u$ with label in $L(u)\cap A^c$ give edges going out of $u\in G'$ with the same label. Notice that $L(v)\cap A$ and $L(u)\cap A^c$ are disjoint, so we can not get two vertices with the same label in this way. As in case $1$, the edges going out of $u\in G$ with label in $L(u)\cap A$ give edges labeled $a$ going out of $u\in\subd{\varphi}{G}$; all the edges labeled $a$ going out of $u\in\subd{\varphi}{G}$ are folded together in $G'$, meaning that there is at most one edge labeled $a$ going out of $u\in G'$. Thus in this case the vertex $u\in G'$ has at most one edge with each label going out of it.

CASE 3: Suppose we have a vertex $u\in G$ with $L(u)\cap A\not=\emptyset$ and $a\not\in L(u)$. This means that all the vertices which are created near $u$ fold together in a single vertex $u_1\in G'$. Each edge going out of $u\in G$ with label in $L(u)\cap A$ gives one edge going out of $u_1\in G'$ with the same label. The vertex $u_1$ also has one edge labeled $\ol{a}$ going out of it, and notice that $\ol{a}\not\in L(u)\cap A$. It follows that the vertex $u_1\in G'$ has at most one edge with each label going out of it.

Since we examined each vertex of $G'$, we conclude that no folding operation is possible on $G'$. By proposition \ref{folding}, we have that each folding sequence starting from $\subd{\varphi}{G}$ can only contain rank-preserving folding operations involving edges labeled $a$, as desired.
\end{proof}

\subsection*{Whitehead's algorithm for subgroups}

We are now ready to state the analogues of theorems \ref{cutvertex}, \ref{Whitehead} and \ref{finesse} for free factors.

\begin{mythm}\label{cutvertex2}
Let $H\sgr F_n$ be a free factor, and suppose $\core{H}$ has more than one vertex. Then the Whitehead graph of $H$ contains a cut vertex.
\end{mythm}

\begin{mythm}\label{Whitehead2}
Let $H\sgr F_n$ be a free factor, and suppose the Whitehead graph of $H$ contains a cut vertex. Then there is a Whitehead automorphism $\varphi$ such that $\core{\varphi(H)}$ has strictly fewer vertices and strictly fewer edges than $\core{H}$.
\end{mythm}

In the next theorem we write $L(v)$ as introduced in definition \ref{Letters}:

\begin{mythm}\label{finesse2}
The automorphism $\varphi=(A,a)$ in theorem \ref{Whitehead2} can be chosen in such a way that, at each vertex $v$ of $\core{H}$, exactly one of the following configurations takes place:

(i) $L(v)\cap A=\emptyset$.

(ii) $L(v)\subseteq A$.

(iii) $a\in L(v)$ and $L(v)\subseteq A\cup\{a\}$.
\end{mythm}

\begin{myrmk}
Case (i) means that we do not act on any of the letters at $v$. Case (ii) means that we act on all the letters at $v$. Case (iii) means that we act on all the letters at $v$ except for $a$.
\end{myrmk}

\begin{myrmk}
Notice that, if $\ol{a}\in L(v)$, then $v$ necessarily falls into case (i).
\end{myrmk}

The proof of theorem \ref{cutvertex2} is analogous to the proof of theorem \ref{cutvertex}.

\begin{proof}[Proof of theorem \ref{cutvertex2}]
Let $H$ be a free factor such that $\core{H}$ has more than one vertex. Up to conjugation, we can assume that the basepoint belongs to $\core{H}$. We also assume that $\core{H}$ contains each letter $x_1,...,x_n$ at least once; otherwise, if $\core{H}$ only contains the letters $x_1,...,x_k$, then we can just apply the same argument in the free factor $\gen{x_1,...,x_k}\sgr\gen{x_1,...,x_n}=F_n$ (using corollary \ref{freefactoromitambient}).

Since $H$ is a free factor, we can take a basis for $H$ and add reduced words $w_1,...,w_r$ in order to make it a basis for $F_n$. Take the graph $\core{H}$ and add $r$ paths from the basepoint to itself, corresponding to the words $w_1,...,w_r$, in order to get a graph $G$. Then, apply a sequence of folding operations $G\rar G'\rar...\rar G^{(l)}$ until no further folding operation is possible, as in proposition \ref{folding}. Since $\gen{H,w_1,...,w_r}=F_n$, we must have that $G^{(l)}=R_n$ is the standard $n$-rose. Using lemma \ref{novalence1}, we can see that no graph in the sequence contains any valence-$1$ vertex. Also, since $\pi_1(G)$ has the same rank as $\pi_1(R_n)$, we must have that each folding operation is rank-preserving.

Thus $G^{(l-1)}$ has no valence-$1$ vertex, and produces the standard $n$-rose with just one rank-preserving folding operation. It is easy to see that $G^{(l-1)}$ has to be of the form described in figure \ref{prerosa}, for some $1\le\alpha\le\beta\le n$ with $\alpha<n$ (up to permutation of the letters, and up to substitution of some letter with its inverse) (and the two edges labeled $x_1$ are the ones to be folded in order to obtain the $n$-rose).

We have a map of graphs $f:\core{H}\rar G^{(l-1)}$ which preserves orientations and labels of edges. The image of $f(\core{H})\subseteq G^{(l-1)}$ contains each letter at least once, meaning that it has to cross at least one of the edges connecting $v$ to $v'$ (see figure \ref{prerosa}). If it crosses the edge labeled $x_1$, then $\ol{x_1}$ is a cut vertex for the Whitehead graph of $H$. If it does not cross the edge labeled $x_1$, then it has to cross the edge $x_{\alpha+i}$ (for some $1\le i\le\beta$), and thus any of $x_{\alpha+i},\ol{x_{\alpha+i}}$ is a cut vertex for the Whitehead graph of $H$.
\end{proof}

Theorems \ref{Whitehead2} and \ref{finesse2} are a consequence of theorem \ref{cutvertex2}.

\begin{proof}[Proof of theorems \ref{Whitehead2} and \ref{finesse2}]
Let $a$ be a cut vertex in the Whitehead graph of $H$.

If the connected component of $a$ does not contain $\ol{a}$, then we take the set $A$ to be that connected component (excluding $a$ itself). Otherwise, take the connected component of $a$ and remove $a$ itself: we remain with at least two nonempty connected components, and at least one of these components does not contain $\ol a$; take $A$ to be such a component. We consider the Whitehead automorphism $\varphi=(A,a)$.

Take a vertex $v$ in $\core{H}$, and notice that the letters in $L(v)$ are vertices of a complete subgraph of the Whitehead graph of $H$. Thus $L(v)$ has to be contained either in $A\cup\{a\}$ or in $A^c$. This yields the tricothomy of theorem \ref{finesse2}.

We now examine more in detail what happens in each of the three cases. The folding takes place according to lemma \ref{onlya2}. For each vertex $v$ of $\core{H}$, we look at the vertices which are created near $v$ in $\subd{\varphi}{\core{H}}$.

Case (i): $L(v)\subseteq A^c$. This means no new vertex is created near $v$. The total number of vertices remains unchanged.

Case (ii): $L(v)\subseteq A$. This means that, for every edge with endpoint $v$, a new vertex is created near $v$. All these new vertices are then folded together into a vertex $v_1$. The vertex $v$ becomes a valence-$1$ vertex, and can thus be removed from the graph. Thus we lose the vertex $v$ and we gain the vertex $v_1$ in the core graph: the total number of vertices is unchanged.

Case (iii): $a\in L(v)$ and $L(v)\subseteq A\cup\{a\}$. This means that $\core{H}$ contains an edge $e$ labeled $a$ going from $v$ to $u$. For every other edge with endpoint $v$, a new vertex is created near $v$. All these new vertices are then folded together with the vertex $u$. The vertex $v$ becomes a valence-$1$ vertex, and can thus be removed from the graph. The total number of vertices decreases by $1$.

In each of the cases (i), (ii) and (iii), the number of vertices and edges of the core graph does not increase. Also, since the Whitehead graph contains at least an edge between $a$ and $A$, we have that case (iii) happens at least once, giving a strict decrease in the number of vertices and edges. This yields theorem \ref{Whitehead2}.
\end{proof}

\begin{myrmk}
We notice that, if $\core{H}$ has rank $r$, the number of edges of $\core{H}$ is the number of vertices plus $r$. The same holds for $\core{\varphi(H)}$, which has rank $r$ too. Thus the decrease in the number of vertices is the same as the decrease in the number of edges.
\end{myrmk}

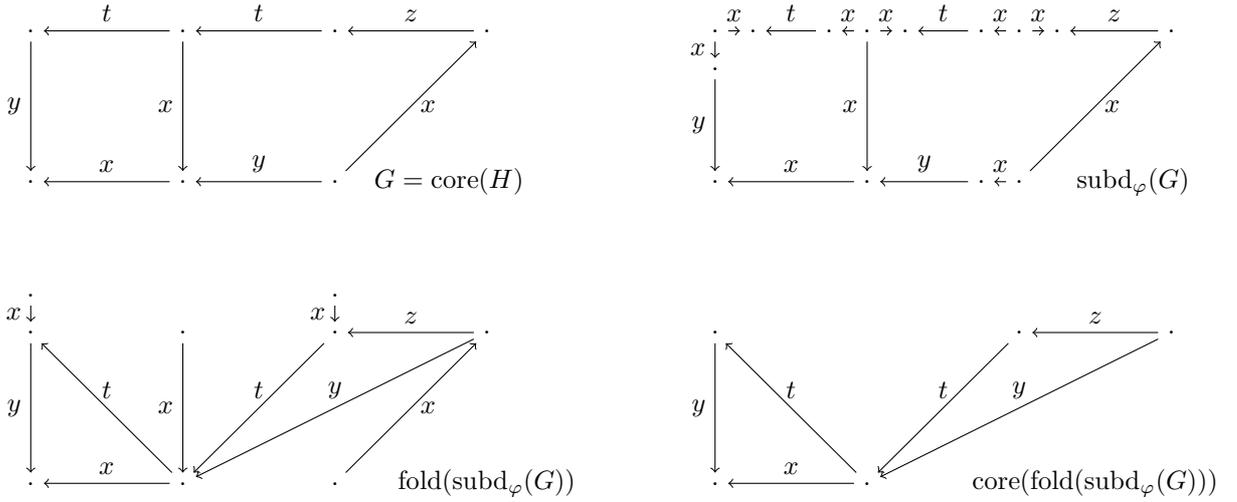
\begin{figure}[h!]
\centering
\begin{tikzpicture}
\node (1) at (3,3) {.};
\node (2) at (3,1) {.};
\node (3) at (1,1) {.};
\node (4) at (1,3) {.};
\draw[->] (1) to node[left]{$x$} (2);
\draw[->] (2) to node[above]{$x$} (3);
\draw[->] (4) to node[left]{$y$} (3);
\draw[->] (1) to node[above]{$t$} (4);
\node (5) at (5,1) {.};
\node (6) at (7,3) {.};
\node (7) at (5,3) {.};
\draw[->] (5) to node[above]{$y$} (2);
\draw[->] (5) to node[right]{$x$} (6);
\draw[->] (6) to node[above]{$z$} (7);
\draw[->] (7) to node[above]{$t$} (1);
\node () at (6.5,1) {$G=\core{H}$};

\node (1a) at (12,3) {.};
\node (2a) at (12,1) {.};
\node (3a) at (10,1) {.};
\node (4a) at (10,3) {.};
\node (8a) at (10,2.5) {.};
\node (9a) at (11.5,3) {.};
\node (10a) at (10.5,3) {.};
\draw[->] (1a) to node[left]{$x$} (2a);
\draw[->] (2a) to node[above]{$x$} (3a);
\draw[->] (4a) to node[left]{$x$} (8a);
\draw[->] (8a) to node[left]{$y$} (3a);
\draw[->] (1a) to node[above]{$x$} (9a);
\draw[->] (9a) to node[above]{$t$} (10a);
\draw[->] (4a) to node[above]{$x$} (10a);
\node (5a) at (14,1) {.};
\node (6a) at (16,3) {.};
\node (7a) at (14,3) {.};
\node (11a) at (13.5,1) {.};
\node (12a) at (14.5,3) {.};
\node (13a) at (13.5,3) {.};
\node (14a) at (12.5,3) {.};
\draw[->] (5a) to node[above]{$x$} (11a);
\draw[->] (11a) to node[above]{$y$} (2a);
\draw[->] (5a) to node[right]{$x$} (6a);
\draw[->] (6a) to node[above]{$z$} (12a);
\draw[->] (7a) to node[above]{$x$} (12a);
\draw[->] (7a) to node[above]{$x$} (13a);
\draw[->] (13a) to node[above]{$t$} (14a);
\draw[->] (1a) to node[above]{$x$} (14a);
\node () at (15.5,1) {$\subd{\varphi}{G}$};

\node (1b) at (3,-1) {.};
\node (2b) at (3,-3) {.};
\node (3b) at (1,-3) {.};
\node (4b) at (1,-0.5) {.};
\node (8b) at (1,-1) {.};
\draw[->] (1b) to node[left]{$x$} (2b);
\draw[->] (2b) to node[above]{$x$} (3b);
\draw[->] (4b) to node[left]{$x$} (8b);
\draw[->] (8b) to node[left]{$y$} (3b);
\draw[->] (2b) to node[above]{$t$} (8b);
\node (5b) at (5,-3) {.};
\node (6b) at (7,-1) {.};
\node (7b) at (5,-0.5) {.};
\node (12b) at (5,-1) {.};
\draw[->] (6b) to node[above]{$y$} (2b);
\draw[->] (5b) to node[right]{$x$} (6b);
\draw[->] (6b) to node[above]{$z$} (12b);
\draw[->] (7b) to node[left]{$x$} (12b);
\draw[->] (12b) to node[above]{$t$} (2b);
\node () at (7,-3) {$\fold{\subd{\varphi}{G}}$};

\node (2c) at (12,-3) {.};
\node (3c) at (10,-3) {.};
\node (8c) at (10,-1) {.};
\draw[->] (2c) to node[above]{$x$} (3c);
\draw[->] (8c) to node[left]{$y$} (3c);
\draw[->] (2c) to node[above]{$t$} (8c);
\node (6c) at (16,-1) {.};
\node (12c) at (14,-1) {.};
\draw[->] (6c) to node[above]{$y$} (2c);
\draw[->] (6c) to node[above]{$z$} (12c);
\draw[->] (12c) to node[above]{$t$} (2c);
\node () at (15,-3) {$\core{\fold{\subd{\varphi}{G}}}$};
\end{tikzpicture}
\caption{Here $F_4=\gen{x,y,z,t}$ and we consider the free factor $H=\gen{ty\ol{x}^2,x\ol{y}xzt}$. The Whitehead transformation $\varphi=(\{y,\ol{z},t,\ol{t}\},x)$ satisfies the tricothomy of theorem \ref{finesse2} for the graph $\core{H}$. In the figure we start with $\core{H}$, we subdivide it, we fold the result and we remove the valence-$1$ vertices: the result is $\core{\varphi(H)}$. \\ We observe that $\core{\varphi(H)}$ can be obtained from $\core{H}$ in the following way: take the vertices of $\core{H}$ which fall in case (iii) of the tricothomy, and collapse to a point the $x$-edges at those vertices; theorem \ref{quotient} makes this formal.}\label{Whiteheadexample}
\end{figure}

\subsection*{The quotient map}

In the following, we use the notation $L(v)$ as introduced in definition \ref{Letters}.

\begin{mydef}\label{fine}
Let $H\sgr F_n$ be a finitely generated non-trivial subgroup and let $\varphi=(A,a)$ be a Whitehead automorphism. We say that the action of $\varphi$ on $H$ is \textbf{\fine} if for each vertex $v\in\core{H}$, exactly one of the following configurations takes place:

(i) $L(v)\cap A=\emptyset$.

(ii) $L(v)\subseteq A$.

(iii) $a\in L(v)$ and $L(v)\subseteq A\cup\{a\}$.
\end{mydef}

\begin{myrmk}
This is exactly the property given by the trichotomy of theorem \ref{finesse2}.
\end{myrmk}

Suppose now the action of $\varphi$ on $H$ is $\fine$. Let $v$ be a vertex of $\core{H}$ that falls in case (iii) of the trichotomy: since $a\in L(v)$, there is a unique edge labeled $a$ going out of $v$. For each vertex $v$ of $\core{H}$ that falls in case (iii), collapse that $a$-edge to a single point. We obtain a quotient $\grafo$ $Q$ together with a quotient map $q:\core{H}\rar Q$.

\begin{mythm}\label{quotient}
There is an isomorphism $\theta:Q\rar\core{\varphi(H)}$ of graphs sending each edge to an edge with the same label and orientation.
\end{mythm}

Before proving the above theorem, we need to introduce another map first. According to proposition \ref{subdivision}, we have $\core{\varphi(H)}=\core{\fold{\subd{\varphi}{\core{H}}}}$. Consider the map $r_1:\core{H}\rar\subd{\varphi}{\core{H}}$ which sends each edge $e$ of $\core{H}$ to the edge-path $\subd{\varphi}{e}$. Consider also the map $r_2:\subd{\varphi}{\core{H}}\rar\fold{\subd{\varphi}{\core{H}}}$ which is given by the quotient map induced by the folding operations. Consider finally the map $r_3:\fold{\subd{\varphi}{\core{H}}}\rar\core{\fold{\subd{\varphi}{\core{H}}}}$ given by the retraction which collapses each edge with a valence-$1$ endpoint to the other endpoint. The composition of these three maps gives a map $r=r_3\circ r_2\circ r_1:\core{H}\rar\core{\varphi(H)}$.

\begin{mythm}\label{quotientmap}
The isomorphism $\theta:Q\rar\core{\varphi(H)}$ in theorem \ref{quotient} can be chosen in such a way that:

(i) For each vertex $v$ of $\core{H}$ we have $r(v)=(\theta\circ q)(v)$.

(ii) For each edge $e$ of $\core{H}$, the map $\restr{r}{e}$ is a (weakly monotone) reparametrization of $\restr{\theta\circ q}{e}$.

In particular, $r$ and $\theta\circ q$ are homotopic relative to the $0$-skeleton of $\core{H}$.
\end{mythm}

See also figure \ref{triangleofmaps}.

\begin{proof}[Proof of theorems \ref{quotient} and \ref{quotientmap}]
We examine cases (i), (ii), (iii) of the trichotomy of definition \ref{fine}.

Let $v$ in $\core{H}$ be a vertex which falls into case (i). Then we have that no new vertex is created near $v$, and the core graph remains unchanged.

Let $v$ in $\core{H}$ be a vertex which falls into case (ii). Then we have that, for each edge at $v$, a new vertex is created near $v$. All these new vertices fold together into a new vertex $v_1$, and $v$ becomes a valence-$1$ vertex and is thus removed from the core graph. The vertex $v_1$ takes the place of the vertex $v$, and the core graph doesn't change.

Let $v$ in $\core{H}$ be a vertex which falls into case (iii). We consider the unique edge $e$ labeled $a$ and going from $v$ to another vertex $u$. For every other edge at $v$, we have that a new vertex is created near $v$. All these new vertices are then folded together and with $u$, and the vertex $v$ becomes a valence-$1$ vertex, and is thus removed from the core graph. The effect on the core graph is exactly the same as collapsing the edge $e$ to a single point.

This shows that the quotient graph $Q$ is isomorphic to $\core{\varphi(H)}$, yielding theorem \ref{quotient}.

Take now an edge $e$ of $\core{H}$ with endpoints $u,v$. Notice that, if $e$ gets collapsed by the quotient map $q$, then it is collapsed by the map $r$ too, and the thesis holds; so assume this is not the case. If both $u$ and $v$ fall into case (i) of the trichotomy, then the quotient map $q$ and the map $r$ send $e$ homeomorphically onto the same edge $q(e)=r(e)$ of $\core{\varphi(H)}$. If $u$ falls into case (i) but $v$ falls into case (ii) or (iii), then $q$ sends $e$ homeomorphically onto an edge $q(e)$ of $\core{\varphi(H)}$. The map $r_1$ maps $e$ to an edge path containing two edges, and the map $r_3$ collapses one of those two edges to a point, and sends the other homeomorphically onto $q(e)$. This yields the conclusion for the edge $e$. The case where $u$ falls into case (ii) or (iii) is completely analogous.
\end{proof}

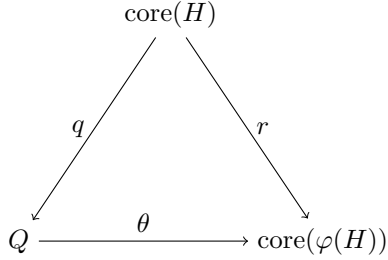
\begin{figure}[h!]
\centering
\begin{tikzpicture}
\node (1) at (0,3) {$\core{H}$};
\node (2) at (-2,0) {$Q$};
\node (3) at (2,0) {$\core{\varphi(H)}$};
\draw[->] (1) to node[left]{$q$} (2);
\draw[->] (1) to node[right]{$r$} (3);
\draw[->] (2) to node[above]{$\theta$} (3);
\end{tikzpicture}
\caption{The maps $q,r$ and $\theta$.}\label{triangleofmaps}
\end{figure}

For brevity, in the following we write $\ol{q}=\theta\circ q:\core{H}\rar\core{\varphi(H)}$. The above theorems \ref{quotient} and \ref{quotientmap} have several interesting consequences.

\begin{mylemma}\label{precise}
Let $H\sgr F_n$ be a non-trivial finitely generated subgroup, and let $\varphi=(A,a)$ be a Whitehead automorphism such that the action of $\varphi$ on $H$ is $\fine$. If case (iii) takes places for exactly $p\ge1$ vertices $v\in\core{H}$, then $\core{\varphi(H)}$ has exactly $p$ fewer vertices and $p$ fewer edges than $\core{H}$. If case (iii) never happens, then $\core{\varphi(H)}=\core{H}$ and the restriction of $\varphi$ to $H$ is conjugation by some element $u\in F_n$.
\end{mylemma}
\begin{proof}
The map $\ol{q}$ collapses exactly one edge for each vertex of $\core{H}$ falling in case (iii). This proves the first part of the proposition.

For the second part, suppose that case (iii) never happens for a vertex of $\core{H}$. This means that the map $\ol{q}:\core{H}\rar\core{\varphi(H)}$ is an isomorphism of $\grafos$.

Consider the pointed core graph $\bcore{H}$ and let $\sigma$ be the (possibly trivial) shortest path from the basepoint to a vertex $v$ of $\core{H}$; moreover, call $t$ the word that you read while going along $\sigma$, from the basepoint to $v$.

Take an element $w\in H$ and think of the corresponding (reduced) path $\alpha$ in $\bcore{H}$ from the basepoint to itself. The path $\alpha$ consists of $\sigma$ followed by $\beta$ followed by the reverse of $\sigma$, for some (reduced) path $\beta$ in $\core{H}$ from $v$ to itself. This gives a decomposition $w=tw'\ol{t}$, where $w'$ is the word that we read while going along $\beta$.

Notice that $\ol{q}$ sends $\beta$ isomorphically onto $\ol{q}(\beta)$, preserving labels and orientation on the edges. If $v$ falls into case (i) of the trichotomy, then this gives $\varphi(w')=w'$, meaning that $\varphi$ acts on $H$ as the conjugation by $\varphi(t)\ol{t}$. If $v$ falls into case (ii) or (iii) of the trichotomy, then this gives $\varphi(w')=aw'\ol{a}$, meaning that $\varphi$ acts on $H$ as the conjugation by $\varphi(t)a\ol{t}$.
\end{proof}

We now show that the trichotomy of definition \ref{fine} has a nice behaviour when we pass to subgroups.

\begin{mylemma}\label{finesubgroups}
Let $K\sgr H\sgr F_n$ be non-trivial finitely generated subgroups, and let $\varphi=(A,a)$ be a Whitehead automorphism. If the action of $\varphi$ on $H$ is $\fine$, then the action of $\varphi$ on $K$ is $\fine$.
\end{mylemma}

\begin{proof}[Proof of lemma \ref{finesubgroups}]
It is enough to notice that for every vertex $u\in\core{K}$ we have that its image $i_*(u)=v\in\core{H}$ satisfies $L(u)\subseteq L(v)$. Since the vertex $v$ satisfies the tricothomy of definition \ref{fine}, so does $u$.
\end{proof}

\begin{myrmk}
In the hypothesis of lemma \ref{finesubgroups}, we have that, for each subgroup $K\sgr H$, the automorphism $\varphi$ either strictly decreases the size of $\core{K}$, or it acts on $K$ as a conjugation by an element of $F_n$. We can actually be more precise than just that. Consider the map $\ol{q}:\core{H}\rar\core{\varphi(H)}$, and let $G\subseteq\core{H}$ be the subgraph given by the union of all the edges which are not collapsed by $q$. Observe that the inclusion $i:K\rar H$ induces a locally injective label-preserving map of graphs $i_*:\core{K}\rar\core{H}$. Then, $\varphi$ acts on $K$ as a conjugation automorphism if and only if $i_*(\core{K})\subseteq G$.
\end{myrmk}

We conclude this section with a technical lemma which will be useful to us later. Let again $K\sgr H\sgr F_n$ be finitely generated non-trivial subgroups. Let $i:K\rar H$ be the inclusion, consider the map of graphs $i_*:\core{H}\rar\core{K}$ and consider the subgraph $i_*(\core{K})\subseteq\core{H}$. For an automorphism $\varphi:F_n\rar F_n$, let $j:\varphi(K)\rar\varphi(H)$ be the inclusion, let $j_*:\core{\varphi(H)}\rar\core{\varphi(K)}$ be the corresponding map of graphs, and consider the subgraph $j_*(\core{\varphi(K)})\subseteq\core{\varphi(H)}$.

\begin{mylemma}\label{fineimage}
Let $K\sgr H\sgr F_n$ be non-trivial finitely generated subgroups. Let $\varphi=(A,a)$ be a Whitehead automorphism such that the action of $\varphi$ on $K$ is $\fine$. Then, with the above notation, $j_*(\core{\varphi(K)})$ has at most as many edges as $i_*(\core{K})$.
\end{mylemma}
\begin{proof}
Let $\ol{q}=\theta\circ q:\core{K}\rar\core{\varphi(K)}$ be as in theorem \ref{quotient}. Suppose we have two edges $e,e'$ in $\core{K}$ such that their image is the same edge $i_*(e)=i_*(e')$ of $\core{H}$, and suppose $\ol{q}$ does not collapse either of $e,e'$. Then the two edges $\ol{q}(e),\ol{q}(e')$ of $\core{\varphi(K)}$ are sent to the same edge $j_*(\ol{q}(e))=j_*(\ol{q}(e'))$ of $\core{\varphi(H)}$. We now divide the edges of $K$ in equivalence classes $E_1,...,E_\alpha$, where each equivalence class is the set of edges with a given image in $\core{H}$ (and in particular, the image $i_*(\core{K})$ has exactly $\alpha$ edges); similarly, we divide the edges of $\core{\varphi(K)}$ in equivalence classes $F_1,...,F_\beta$, based on their image in $\core{\varphi(H)}$. Then each $F_j$ is a union of some $E_i$s, implying that $\beta\le\alpha$, as desired.
\end{proof}

\begin{myrmk}
Lemma \ref{fineimage} becomes false if we try to count the number of vertices, instead of counting the number of edges.
\end{myrmk}

\subsection*{A relative version of Whithead's algorithm}

Let $F_n=\gen{x_1,...,x_n}$ and consider the free factor $\gen{x_1,...,x_k}$ for $1\le k\le n-1$.

\begin{mythm}\label{Whitehead4}
Let $w\in F_n$ be primitive and not a single letter. Suppose there is an automorphism $\theta:F_n\rar F_n$ such that $\theta(\gen{x_1,...,x_k})=\gen{x_1,...,x_k}$ and $\theta(w)=x_{k+1}$. Then there is a Whitehead automorphism $\varphi=(A,a)$ such that:

(i) $\varphi(x_i)=x_i$  for $i=1,...,k$.

(ii) The length of $\varphi(w)$ is strictly smaller than the length of $w$.

(iii) Every letter $a$, which is added to $w$ when applying $\varphi$ to $w$ letter-by-letter, immediately cancels (in the free reduction process).
\end{mythm}

\begin{myrmk}
Notice that the word $w$ is not required to be cyclically reduced. In (ii) we mean the length and not the cyclic length. In (iii) we consider the free reduction process and not the cyclic reduction process.
\end{myrmk}

For the proof, we need the following straightforward lemma:

\begin{mylemma}\label{lemmino}
Consider the inner automorphism $\gamma_a(w)=aw\ol{a}$ of $F_n$. Then for every Whitehead automorphism $(A,a)$, the identity $(A,a)=\gamma_a\circ(A^c\setminus\{a,\ol a\},\ol a)$ holds.
\end{mylemma}

\begin{proof}[Proof of theorem \ref{Whitehead4}]
Consider the free factor $H=\gen{x_1,...,x_k}*\gen{w}$. Notice that $\core{H}$ consists of $\bcore{w}$ together with $k$ edges from the basepoint to itself, labeled with the letters $x_1,...,x_k$ (and here it is important that $k\ge1$). We apply theorems \ref{Whitehead2} and \ref{finesse2} to $H$ in order to get a Whitehead automorphism $\varphi=(A,a)$. If the basepoint of $\core{H}$ would fall into case (ii) or (iii) of the tricothomy of theorem \ref{finesse2}, then we apply lemma \ref{lemmino} and consider the Whitehead automorphism $\varphi=(A^c\setminus\{a,\ol a\},\ol a)$ instead. Then $\varphi$ satisfies all of the desired properties.
\end{proof}

We observe that theorem \ref{Whitehead4} can also be generalized to subgroups (and the proof is the same, so will be omitted).

\begin{mythm}\label{Whitehead5}
Let $H\sgr F_n$ be a free factor of rank $r\ge1$ and suppose that $\bcore{H}$ has at least two vertices. Suppose there is an automorphism $\theta:F_n\rar F_n$ such that $\theta(\gen{x_1,...,x_k})=\gen{x_1,...,x_k}$ and $\theta(H)=\gen{x_{k+1},...,x_{k+r}}$. Then there is a Whitehead automorphism $\varphi=(A,a)$ such that:

(i) $\varphi(x_1)=x_1$ and ... and $\varphi(x_k)=x_k$.

(ii) The graph $\bcore{\varphi(H)}$ has strictly fewer vertices and edges than $\bcore{H}$.

(iii) The trichotomy of theorem \ref{finesse2} holds at each vertex $v\in\bcore{H}$. Moreover, the basepoint always falls into case (i) of the tricothomy.
\end{mythm}

\section{About computation of distances in the complex of free factors}\label{FreeFactorComplex}

For an element $w\in F_n$, we denote by $[w]$ the conjugacy class of that element. For a subgroup $H\sgr F_n$, we denote by $[H]$ the conjugacy class of that subgroup.

We are now going to define a simplicial complex $\FF{n}$, starting with its $0$-skeleton and its $1$-skeleton.  The $0$-skeleton $\FF{n}^0$ has a point $[H]$ for every conjugacy class of free factors $H\sgr F_n$. The $1$-skeleton $\FF{n}^1$ is defined as follows: add a $1$-simplex with vertices $[H_0],[H_1]$ if and only if $[H_0]\not=[H_1]$ and there are representatives $H_0'\in[H_0]$ and $H_1'\in[H_1]$ and a permutation $\sigma:\{0,1\}\rar\{0,1\}$ such that $H_{\sigma(0)}'\sgr H_{\sigma(1)}'$. Define $\FF{n}$ as the flag complex over the $1$-skeleton $\FF{n}^1$: we have a $k$-simplex with endpoints $[H_0],...,[H_k]$ if and only if $[H_0],...,[H_k]$ are pairwise connected by $1$-simplices in $\FF{n}^1$. Equivalently, we have a $k$-simplex with endpoints $[H_0],...,[H_k]$ if and only if $[H_0],...,[H_k]$ are pairwise distinct and there are representatives $H_0'\in[H_0],...,H_k'\in[H_k]$ and a permutation $\sigma:\{0,...,k\}\rar\{0,...,k\}$ such that $H_{\sigma(0)}'\sgr...\sgr H_{\sigma(k)}'$.

\begin{mydef}
The simplicial complex $\FF{n}$ defined above is called \textbf{complex of free factors}.
\end{mydef}

It is shown in \cite{Bestvina} that $\FF{n}$ is connected. We would like to determine whether there is an algorithm that, given two vertices of $\FF{n}$, gives as output their distance in a finite time, where distance is combinatorial distance in the $1$-skeleton $\FF{n}^1$. We here furnish algorithms for distances $1,2,3$, and also an algorithm for distance $4$ when one of the free factors has rank $n-1$.

\subsection*{Distance one}

It is easy to check whether two conjugacy classes of free factors $[H],[K]$ are at distance $1$ or not. Assume $\rank{H}\ge\rank{K}$. We look for representatives $H'\in[H]$ and $K'\in[K]$ with an inclusion $K'\sgr H'$. This is equivalent to looking for a locally injective map of graphs $\core{K}\rar\core{H}$. Each such map, if it exists, is uniquely determined by the image of a given vertex; thus we only have to deal with a finite number of tries.


\subsection*{Distance two}

We will rely on the following proposition:

\begin{myprop}\label{distance2}
Let $H,K$ be non-trivial free factors, and suppose that $\core{H}\sqcup\core{K}$ contains at least one edge with each label. Suppose there are free factors $H'\in[H]$ and $K'\in[K]$ and $J\not=F_n$ such that $H',K'\sgr J$. Then there is a Whitehead automorphism $\varphi=(A,a)$ such that $\core{\varphi(H)}\sqcup\core{\varphi(K)}$ has strictly fewer vertices and strictly fewer edges than $\core{H}\sqcup\core{K}$.
\end{myprop}
\begin{proof}
Since $J$ is a free factor, by a recursive application of theorems \ref{Whitehead2} and \ref{finesse2}, we obtain a chain of Whitehead automorphisms $\varphi_1,...,\varphi_l$ such that $\core{\varphi_l\circ...\circ\varphi_1(J)}$ is a rose with labels only in $\{x_1,...,x_{n-1}\}$. By lemmas \ref{finesubgroups} and \ref{precise}, we have that either $\core{\varphi_1(H)}\sqcup\core{\varphi_1(K)}=\core{H}\sqcup\core{K}$ or $\core{\varphi_1(H)}\sqcup\core{\varphi_1(K)}$ has strictly fewer vertices and strictly fewer edges than $\core{H}\sqcup\core{K}$. If $\core{\varphi_1(H)}\sqcup\core{\varphi_1(K)}=\core{H}\sqcup\core{K}$ then we repeat the reasoning with $\varphi_2$ instead of $\varphi_1$; and so on. If $\core{\varphi_l\circ...\circ\varphi_1(H)}\sqcup\core{\varphi_l\circ...\circ\varphi_1(K)}=\core{H}\sqcup\core{K}$ then we have a contradiction, since $\core{\varphi_l\circ...\circ\varphi_1(H)}\sqcup\core{\varphi_l\circ...\circ\varphi_1(K)}$ only contains edges with the labels $\{x_1,...,x_{n-1}\}$ while $\core{H}\sqcup\core{K}$ contains edges with all possible labels by hypothesis. So we can take the smallest $m$ such that $\core{\varphi_m(H)}\sqcup\core{\varphi_m(K)}\not=\core{H}\sqcup\core{K}$, and the Whitehead automorphism $\varphi_m$ satisfies the thesis.
\end{proof}

Let $[H],[K]$ be conjugacy classes of non-trivial free factors. We want to check (i) whether or not there are representatives with non-trivial intersection and (ii) whether or not there are representatives contained in a common proper free factor.

For (i) it is possible to use the technique explained in \cite{Stallings}. There are representatives $H'\in[H]$ and $K'\in[K]$ with non-trivial intersection if and only if the pullback of the two graphs $\core{H}$ and $\core{K}$ contains a non-trivial cycle.

For (ii) we apply proposition \ref{distance2} repeatedly. If $\core{H}\sqcup\core{K}$ contains only edges with labels from $\{x_1,...,x_{n-1}\}$ (or from any other proper subset of $\{x_1,...,x_n\}$) then they are at distance two; otherwise we look for a Whitehead transformation which strictly reduces the number of vertices of $\core{H}\sqcup\core{K}$: if we don't find it, then they are not at distance two, if we do then we apply it and reiterate the reasoning.

\subsection*{Distance three}

Given two conjugacy classes of free factors $[H],[K]$, we want to check whether there are representatives $H'\in[H]$ and $K'\in[K]$ and non-trivial free factors $I,J$ such that $H',J\sgr I$ and $J\sgr K'$ (there is also a symmetric check to do, but it is completely analogous).

Define the finite oriented graph $\Theta$ as follows. The graph $\Theta$ has one vertex corresponding to each pair of core folded $\grafos$ $(A,B)$ such that $A$ has at most as many edges as $\core{H}$ and $B$ has at most as many edges as $\core{K}$. There is an oriented edge from $(A,B)$ to $(C,D)$ if and only if there is a Whitehead automorphism $\varphi$ such that $C=\core{\fold{\subd{\varphi}{A}}}$ and $D$ is isomorphic to a subgraph of $\core{\fold{\subd{\varphi}{B}}}$; in that case we label that edge of $\Theta$ with the automorphism $\varphi$.

Suppose now there are non-trivial free factors $I,J$ such that $H',J\sgr I$ and $J\sgr K'$. The inclusion $j:J\rar K'$ gives a locally injective map $j_*:\core{J}\rar\core{K}$, and thus a subgraph $j_*(\core{J})\subseteq\core{K}$: the pair $(\core{H},j_*(\core{J}))$ is a vertex of $\Theta$. If $\core{I}$ contains only edges with labels from $\{x_1,...,x_{n-1}\}$, then we see that the pair of graphs $(\core{H},j_*(\core{J}))$ only contains edges with those labels too. If $\core{I}$ contains at least one edge with each label, then by theorems \ref{Whitehead2} and \ref{finesse2} there is a Whitehead automorphism $\varphi$ such that $\core{\varphi(I)}$ has strictly fewer edges than $\core{I}$, and such that for each vertex in $\core{I}$ the tricothomy of theorem \ref{finesse2} holds. In particular, by lemmas \ref{precise}, \ref{finesubgroups}, \ref{fineimage}, we have that the number of edges of $\core{H}$ and of $j_*(\core{J})$ does not increase either. This means that the pair $(\core{\varphi(H)},j_*(\core{\varphi(J)}))$ is a vertex of $\Theta$, and that $\Theta$ contains an edge labeled $\varphi$ and going from $(\core{H},j_*(\core{J}))$ to $(\core{\varphi(H)},j_*(\core{\varphi(J)}))$ (here we are using proposition \ref{subdivision}).

We now reiterate the same reasoning. By theorems \ref{Whitehead2} and \ref{finesse2}, we can take a finite sequence of Whitehead automorphisms $\varphi_1,...,\varphi_l$ such that $\varphi_i$ strictly reduces the number of edges of $\core{\varphi_{i-1}\circ...\circ\varphi_1(I)}$, such that for each vertex of $\core{\varphi_{i-1}\circ...\circ\varphi_1(I)}$ the tricothomy of theorem \ref{finesse2} holds, and such that $\varphi_l\circ...\circ\varphi_1(I)$ only contains edges with labels from $\{x_1,...,x_{n-1}\}$. Then this produces a path in $\Theta$ with vertices $(\core{\varphi_i\circ...\circ\varphi_1(H)},j_*(\core{\varphi_i\circ...\circ\varphi_1(J)}))$ and which goes from the pair $(\core{H},j_*(\core{J}))$ to a pair containing only edges with labels in $\{x_1,...,x_{n-1}\}$. Since the graph $\Theta$ is finite, there is an algorithm that tells us whether such a path in $\Theta$ exists or not.

Conversely, given two conjugacy classes of free factors $[H],[K]$, suppose there is a path in $\Theta$ with vertices $(A_1,B_1),...,(A_l,B_l)$ and with an edge labeled $\varphi_i$ going from $(A_i,B_i)$ to $(A_{i+1},B_{i+1})$, such that $A_1=\core{H}$ and $B_1$ is a subgraph of $\core{K}$, and such that $A_l,B_l$ only contain edges with labels in $\{x_1,...,x_{n-1}\}$. Then we fix basepoints in $A_l$ and $B_l$ and we set $\psi=\varphi_1^{-1}\circ...\circ\varphi_l^{-1}$: we get a segment of length three in $\FF{n}^1$ connecting $[H]$ and $[K]$, with vertices $[H]=[\psi(\pi_1(A_l))]$ and $[I]=[\psi(\gen{x_1,...,x_{n-1}})]$ and $[J]=[\psi(\pi_1(B_l))]$ and $[K]$.

Thus, given conjugacy classes of free factors $[H]$ and $[K]$, the existence of non-trivial free factors $I,J$ such that $H',J\sgr I$ and $J\sgr K'$ is equivalent to the existence of a path in $\Theta$ from a vertex of the form $(\core{H},B_1)$, with $B_1\subseteq\core{K}$, to a vertex of the form $(A_l,B_l)$, where $A_l\sqcup B_l$ does not use all the labels in $\{x_1,...,x_n\}$. This yields an algorithm to check whether two vertices of $\FF{n}$ are at distance three or not.
%

\subsection*{About distance four}

We would like to check whether two conjugacy classes of free factors $[H],[K]$ are at distance at most four in $\FF{n}$. In order to achieve this, we need to check two conditions:

\begin{enumerate}
\item Whether or not there are representatives $H'\in[H]$ and $K'\in[K]$ and non-trivial free factors $J_1,J_2,J_3$ such that $J_1\sgr H'$ and $J_1,J_3\sgr J_2$ and $J_3\sgr K'$.
\item Whether or not there are representatives $H'\in[H]$ and $K'\in[K]$ and non-trivial free factors $J_1,J_2,J_3$ such that $H',J_2\sgr J_1$ and $J_2,K'\sgr J_3$.
\end{enumerate}

We here furnish an algorithm to check condition 1.

\begin{myrmk}
In the particular case when $\rank{H}=n-1$, condition 2 reduces to checking distance three. In particular, when one of the free factors has rank $n-1$, we have an algorithm to check whether they are at distance four or not.
\end{myrmk}

The technique is the same as for distance three. Consider the oriented graph $\Omega$ defined as follows. We have one vertex for each pair of core folded $\grafos$ $(A,B)$ such that $A$ has at most as many edges as $\core{H}$ and $B$ has at most as many edges as $\core{K}$. There is an oriented edge from $(A,B)$ to $(C,D)$ if and only if there is a Whitehead automorphism $\varphi$ such that $C$ is isomorphic to a subgraph of $\core{\fold{\subd{\varphi}{A}}}$ and $D$ is isomorphic to a subgraph of $\core{\fold{\subd{\varphi}{B}}}$; in that case we label that edge of $\Omega$ with the automorphism $\varphi$.

Suppose there are representatives $H'\in[H]$ and $K'\in[K]$ and non-trivial free factors $J_1,J_2,J_3$ such that $J_1\sgr H'$ and $J_1,J_3\sgr J_2$ and $J_3\sgr K'$. By means of theorems \ref{Whitehead2} and \ref{finesse2}, we take a chain of Whitehead automorphisms $\varphi_1,...,\varphi_l$ such that $\varphi_{i+1}$ strictly reduces the number of edges of $\core{\varphi_i\circ...\circ\varphi_1(J_2)}$, and such that the tricothomy of theorem \ref{finesse2} holds too. By lemmas \ref{precise}, \ref{finesubgroups}, \ref{fineimage}, we have that this produces a path $(A_i,B_i)$ in $\Omega$, where $A_i$ is the image the map $\core{\varphi_i\circ...\circ\varphi_1(J_1)}\rar\core{\varphi_i\circ...\circ\varphi_1(H)}$ induced by the inclusion $J_1\sgr K'$, and $B_i$ is the image of the map $\core{\varphi_i\circ...\circ\varphi_1(J_3)}\rar\core{\varphi_i\circ...\circ\varphi_1(K)}$ induced by the inclusion $J_3\sgr K'$. The starting point $(A_1,B_1)$ of the path is given by two subgraphs of $\core{H}$ and $\core{K}$ respectively, and the endpoint $(A_l,B_l)$ has the property that $A_l\sqcup B_l$ only contains edges with labels from a proper subset of $\{x_1,...,x_n\}$.

Conversely, suppose there is a path $(A_1,B_1),...,(A_l,B_l)$ in $\Omega$ with an edge from $(A_i,B_i)$ to $(A_{i+1},B_{i+1})$ labeled $\varphi_i$, and such that $A_1,B_1$ are subgraphs of $\core{H},\core{K}$ respectively, and $A_l\sqcup B_l$ contains only edges with labels in $\{x_1,...,x_{n-1}\}$. Then we fix basepoints in $A_l$ and $B_l$, we set $\psi=\varphi_1^{-1}\circ...\circ\varphi_l^{-1}$, and we produce the free factors $J_1=\psi(\pi_1(A_l))$ and $J_2=\psi(\gen{x_1,...,x_{n-1}})$ and $J_3=\psi(\pi_1(B_l))$. For these free factors, there are representatives $H'\in[H]$ and $K'\in[K]$ such that $J_1\sgr H'$ and $J_1,J_3\sgr J_2$ and $J_3\sgr K'$, as desired.

Since the graph $\Omega$ is finite, we obtain an algorithm to check condition 1.

\bibliographystyle{alpha}
\nocite{*}
\bibliography{bibliography.bib}

\begin{thebibliography}{Whi36b}

\bibitem[BB21]{Bridson}
M.~Bestvina and M.~R. Bridson.
\newblock {Rigidity of the complex of free factors}, 2021.

\bibitem[BF14]{Bestvina}
M.~Bestvina and M.~Feighn.
\newblock {Hyperbolicity of the Complex of Free Factors}.
\newblock {\em Adv. Math.}, 256:104--155, 2014.

\bibitem[BFH20]{BFH}
M.~Bestvina, M.~Feighn, and M.~Handel.
\newblock {A {M}c{C}ool {W}hitehead Type Theorem for Finitely Generated
  Subgroups of $\mathsf{Out}({F}_n)$}, 2020.

\bibitem[CG10]{Goldstein}
A.~Clifford and R.~Z. Goldstein.
\newblock {Subgroups of Free Groups and Primitive Elements}.
\newblock {\em J. Group Theory}, 13(4):601--611, 2010.

\bibitem[CV86]{Vogtmann}
M.~Culler and K.~Vogtmann.
\newblock {Moduli of Graphs and Automorphisms of Free Groups}.
\newblock {\em Invent. Math.}, 84(1):91--119, 1986.

\bibitem[Dic14]{Dicks}
W.~Dicks.
\newblock {On Free-Group Algorithms that Sandwich a Subgroup Between
  Free-Product Factors}.
\newblock {\em J. Group Theory}, 17(1):13--28, 2014.

\bibitem[Ger84]{Gersten}
S.~M. Gersten.
\newblock {On {W}hitehead's Algorithm}.
\newblock {\em Bull. Amer. Math. Soc. (N.S.)}, 10(2):281--284, 1984.

\bibitem[HL74]{Lyndon}
P.~J. Higgins and R.~C. Lyndon.
\newblock {Equivalence of Elements under Automorphisms of a Free Group}.
\newblock {\em J. London Math. Soc. (2)}, 8:254--258, 1974.

\bibitem[HW19]{Heusener}
M.~Heusener and R.~Weidmann.
\newblock {A Remark on {W}hitehead's Cut-Vertex Lemma}.
\newblock {\em J. Group Theory}, 22(1):15--21, 2019.

\bibitem[LS01]{LyndonSchupp}
R.~C. Lyndon and P.~E. Schupp.
\newblock {\em {Combinatorial Group Theory}}.
\newblock Classics in Mathematics. Springer-Verlag, Berlin, 2001.
\newblock Reprint of the 1977 edition.

\bibitem[Pud14]{Puder}
D.~Puder.
\newblock {Primitive Words, Free Factors and Measure Preservation}.
\newblock {\em Israel J. Math.}, 201(1):25--73, 2014.

\bibitem[Rap58]{Rapaport}
E.~S. Rapaport.
\newblock {On Free Groups and their Automorphisms}.
\newblock {\em Acta Math.}, 99:139--163, 1958.

\bibitem[Sta83]{Stallings}
J.~R. Stallings.
\newblock {Topology of Finite Graphs}.
\newblock {\em Invent. Math.}, 71(3):551--565, 1983.

\bibitem[Sta99]{Stallings2}
J.~R. Stallings.
\newblock {{W}hitehead Graphs on Handlebodies}.
\newblock In {\em Geometric group theory down under ({C}anberra, 1996)}, pages
  317--330. de Gruyter, Berlin, 1999.

\bibitem[Whi36a]{Whitehead}
J.~H.~C. Whitehead.
\newblock {On Certain Sets of Elements in a Free Group}.
\newblock {\em Proc. London Math. Soc. (2)}, 41(1):48--56, 1936.

\bibitem[Whi36b]{Whitehead2}
J.~H.~C. Whitehead.
\newblock {On Equivalent Sets of Elements in a Free Group}.
\newblock {\em Ann. of Math. (2)}, 37(4):782--800, 1936.

\end{thebibliography}

\end{document}